\documentclass[draft]{amsart}
\usepackage{mathrsfs}
\usepackage{amssymb}
\usepackage{amsmath}
\usepackage{amsfonts}
\usepackage{amsfonts,enumerate}
\usepackage{pifont}
\usepackage{enumerate}
\usepackage{graphics}
\usepackage{verbatim}
\usepackage{amsthm}

\numberwithin{equation}{section}
\newtheorem{theorem}{Theorem}[section]

\newtheorem{definition}{Definition}[section]
\newtheorem{lemma}{Lemma}[section]

\newtheorem{remark}{Remark}[section]

\newcommand{\A}{{\mathcal A}}

\newcommand{\g}{\gamma}

\newcommand{\8}{\infty}
\newcommand{\el}{\ell}

\newcommand{\be}{\begin{eqnarray*}}
\newcommand{\ee}{\end{eqnarray*}}
\newcommand{\beq}{\begin{equation}}
\newcommand{\eeq}{\end{equation}}
\newcommand{\beqn}{\begin{equation*}}
\newcommand{\eeqn}{\end{equation*}}
\newcommand{\bs}{\begin{split}}
\newcommand{\es}{\end{split}}

\numberwithin{equation}{section}

\begin{document}
\title{Atomic decomposition of real-variable type for Bergman spaces in the unit ball of $\mathbb{C}^n$}

\thanks{{\it 2010 Mathematics Subject Classification:} 32A36, 32A50.}
\thanks{{\it Key words:} Bergman space, atomic decomposition, homogeneous space, maximal function, Bergman metric.}

\author{Zeqian Chen}

\address{Wuhan Institute of Physics and Mathematics, Chinese Academy of Sciences, West District 30, Xiao-Hong-Shan, Wuhan 430071, China}
\email{zqchen@wipm.ac.cn}


\author{Wei Ouyang}

\address{Institute of Geodesy and Geophysics, Chinese Academy of Sciences, 340 Xu-Dong Road, Wuhan 430077, China
and University of Chinese Academy of Sciences, Beijing 100049, China}


\date{}
\maketitle

\markboth{Z. Chen and W. Ouyang}%
{Bergman spaces}

\begin{abstract}
In this paper, we show that every (weighted) Bergman space $\mathcal{A}^p_{\alpha} (\mathbb{B}_n)$ in the complex ball admits an atomic decomposition of real-variable type for any $0 < p \le 1$ and $\alpha > -1.$ More precisely, for each $f \in \mathcal{A}^p_{\alpha} (\mathbb{B}_n)$ there exist a sequence of real-variable $(p, \8)_{\alpha}$-atoms $a_k$ and a scalar sequence $\{\lambda_k \}$ with $\sum_k | \lambda_k |^p < \8$ such that $f = \sum_k \lambda_k P_{\alpha} ( a_k),$ where $P_{\alpha}$ is the Bergman projection from $L^2_{\alpha} (\mathbb{B}_n)$ onto $\mathcal{A}^2_{\alpha} (\mathbb{B}_n).$ The proof is constructive, and our construction is based on some sharp estimates about Bergman metric and Bergman kernel functions in $\mathbb{B}_n.$
\end{abstract}



\section{Introduction}\label{intro}

Atomic decomposition, initiated by Coifman \cite{Coifman1974}, plays a fundamental role in harmonic analysis. For instance, atomic decomposition is a powerful tool for dealing with duality theorems, interpolation theorems and some fundamental inequalities in harmonic analysis. There are extensive works on the atomic decomposition of Hardy spaces (see \cite{Grafakos2009, Stein1993} and references therein). In the case of several complex variables, Coifman, Rochberg and Weiss (in an unpublished version of \cite{CRW1976}) firstly proved the atomic decomposition theorem for (holomorphic) Hardy spaces on the complex ball when $p=1.$ Subsequently, Garnett and Latter \cite{GL1978} generalized this result to the case $0 < p < 1.$ The associated theorem for strongly pseudoconvex domains in $\mathbb{C}^n$ was proved by Dafni \cite{Dafni1994}, and independently by Krantz and Li \cite{KL1995} wherein the corresponding result for pseudoconvex domains of finite type in $\mathbb{C}^2$ was obtained as well. Furthermore, Grellier and Peloso \cite{GP2002} presented the atomic decomposition theorem for Hardy spaces on convex domains of finite type in $\mathbb{C}^n.$ On the other hand, Coifman and Rochberg \cite{CR1980} obtained decomposition theorems for (weighted) Bergman spaces $\mathcal{A}^p_{\alpha} (\mathbb{B}_n)$ involving `complex-variable' atoms defined by kernel functions in the complex ball for any $0< p<\8$ and $\alpha > -1.$

However, to the best of our knowledge, atomic decomposition (of real-variable type) for Bergman spaces was presented in the literature only in the case $p=1.$ This was given by Coifman and Weiss \cite{CW1977} based on their theory of harmonic analysis on homogenous spaces. Recently, using the duality between Bergman and Bloch spaces, the present authors \cite{CO1} proved an atomic decomposition for Bergman spaces on the complex ball when $p=1,$ in terms of real-variable atoms with respect to Carleson tubes. But the approaches in \cite{CO1, CW1977} are both based on duality and theorefore not constructive and cannot be applied to the case $0< p < 1.$ The aim of this paper is to extend to $0 < p < 1$ these results for $p=1,$ through using an constructive method. Although the strategy behind the proof is in the same spirit as the analogous one of Hardy spaces (e.g. \cite{Dafni1994, GL1978, GP2002, KL1995}), but the technical arguments involved here are different slightly from the ones found there. This is mainly due to the fact that functions in Bergman spaces cannot be identified with those on the boundary of the complex ball, and hence the construction of decomposition must be done in the complex ball. Indeed, our construction is based on a quasi-metric and some sharp estimates about Bergman metric and Bergman kernel functions in the unit ball of $\mathbb{C}^n.$

The paper is organized as follows. In Section \ref{pre} we present preliminaries and in particular introduce local coordinates which reflect the complex structure. In Section \ref{Max} we introduce several maximal functions associated with Bergman spaces, which play a crucial role in the construction of atomic decomposition for Bergman spaces. In Section \ref{result} we define $(p, \8)_{\alpha}$-atoms and state the corresponding atomic decomposition for Bergman spaces. Subsequently, we show that the $(p, \8)_{\alpha}$-atoms so defined are suitable for our purpose. Section \ref{atomicdecomp} is devoted to the construction of the associated atomic decomposition. In section \ref{good-function} we will construct and estimate a collection of smooth bump functions which is crucial in the construction mentioned. Sections \ref{AtomConstruction} and \ref{AtomicdecompPf} are devoted to present the associated construction.

In what follows, $C$ always denotes a constant depending only on $n, \g, p, \alpha$ and $N,$ which may be different in different places. For two nonnegative (possibly
infinite) quantities $X$ and $Y,$ by $X \lesssim Y$ we mean that there exists a positive constant $C$ such that $ X \leq C Y,$ and by $X \thickapprox Y$ that $X \lesssim Y$ and $Y \lesssim X.$ Any notation and terminology not otherwise explained, are as used in \cite{Zhu2005} for spaces of holomorphic functions in the unit ball of $\mathbb{C}^n.$

\section{Preliminaries and notation}\label{pre}

Throughout the paper we fix a positive integer $n \ge 1$ and a parameter $\alpha > -1.$ We denote by $\mathbb{C}^n$ the Euclidean space of complex dimension $n$. For $z=(z_1,\cdots,z_n)$ and $w=(w_1,\cdots,w_n)$ in $\mathbb{C}^n,$ we write
\be
\langle z,w\rangle=z_1\overline{w}_1+\cdots+z_n\overline{w}_n,
\ee
where $\overline{w}_k$ is the complex conjugate of $w_k.$ We also write
\be
|z|=\sqrt{|z_1|^2+\cdots+|z_n|^2}.
\ee
The open unit ball in $\mathbb{C}^n$ is the set
\be
{\mathbb{B}_n} =\{\,z\in {\mathbb{C}^n}: |z|<1\}.
\ee
The boundary of $\mathbb{B}_n$ will be denoted by $\mathbb{S}_n$ and is called the unit sphere in $\mathbb{C}^n,$ i.e.,
\be
{\mathbb{S}_n}=\{\,z\in {\mathbb{C}^n}: |z|=1\}.
\ee

\subsection{Bergman spaces}\label{Bergmanspace}

For $\alpha > -1$ and $p>0$ the (weighted) Bergman space $\mathcal{A}^p_{\alpha}$ consists of holomorphic functions $f$ in $\mathbb{B}_n$ with
$$\|f\|_{p,\,\alpha}=\left ( \int_{\mathbb{B}_n}|f(z)|^pdv_{\alpha}(z) \right )^{\frac{1}{p}}<\infty,$$
where the weighted Lebesgue measure $dv_{\alpha}$ on $\mathbb{B}_n$ is defined by
$$dv_{\alpha}(z)=c_{\alpha}(1-|z|^2)^{\alpha}dv(z)$$
and $c_{\alpha}=\Gamma(n+\alpha+1)/ [n!\Gamma(\alpha+1)]$ is a normalizing constant so that $dv_{\alpha}$ is a probability measure on $\mathbb{B}_n.$ Thus,
\be
\mathcal{A}^p_{\alpha} = \mathcal{H} (\mathbb{B}_n) \cap L^p_{\alpha} (\mathbb{B}_n),
\ee
where $\mathcal{H} (\mathbb{B}_n)$ is the space of all holomorphic functions in $\mathbb{B}_n,$ and $L^p_{\alpha} (\mathbb{B}_n)$ is the usual $L^p$ space on the measure space $(\mathbb{B}_n, d v_{\alpha}).$ When $\alpha =0$ we simply write $\A^p$ for $\A^p_0.$ These are the usual Bergman spaces. Note that for $1 \le p < \8,$ $\mathcal{A}^p_{\alpha}$ is a Banach space under the norm $\|\ \|_{p,\,\alpha}.$ If $0 < p <1,$ the space $\mathcal{A}^p_{\alpha}$ is a quasi-Banach space with $p$-norm $\| f \|^p_{p, \alpha}.$

Recall that $P_{\alpha}$ is the orthogonal projection from $L^2_{\alpha} (\mathbb{B}_n)$ onto ${\mathcal A}^2_{\alpha},$ which can be expressed as
\be
P_{\alpha}f(z)=\int_{\mathbb{B}_n}K^{\alpha}(z,w)f(w)dv_{\alpha}(w), \quad \forall f \in L^1(\mathbb{B}_n,dv_{\alpha}),\; \alpha>-1,
\ee
where
\be
K^{\alpha}(z,w)=\frac{1}{\big(1-\langle z,w\rangle\big)^{n+1+\alpha}},\quad z,w\in\mathbb{B}_n.
\ee
$P_{\alpha}$ extends to a bounded projection from $L^p_{\alpha} (\mathbb{B}_n)$ onto $\mathcal A^p_{\alpha}$ for $1<p<\infty$ (see \cite[Theorem 2.11]{Zhu2005}).

Let $D(z,\gamma)$ denote the Bergman metric ball at $z$
\be
D(z, \gamma) = \{w \in \mathbb{B}_n\;: \; \beta (z, w) < \g \}
\ee
with $\gamma >0,$ where $\beta$ is the Bergman metric on $\mathbb{B}_n.$ It is known that
\be
\beta (z, w) = \frac{1}{2} \log \frac{1 + | \varphi_z (w)|}{1 - | \varphi_z (w)|},\quad z, w \in \mathbb{B}_n,
\ee
whereafter $\varphi _z$ is the bijective holomorphic mapping in $\mathbb{B}_n,$ which satisfies $\varphi _z (0)=z$, $\varphi _z(z)=0$ and
$\varphi _z\circ\varphi _z = id.$

For the sake of convenience, we collect two elementary facts on the Bergman metric and holomorphic functions in the unit ball of $\mathbb{C}^n.$

\begin{lemma}\label{le:BergmanEst01} {\rm (cf. \cite[Lemma 1.24]{Zhu2005})}
For any real $\alpha$ and positive $\g$ there exist constant $C_{\gamma}$ such that
\be
C_{\gamma}^{-1}(1-|z|^2)^{n+1+\alpha}\le v_{\alpha}(D(z,\g))\le C_{\gamma}(1-|z|^2)^{n+1+\alpha}
\ee
for all $z\in\mathbb{B}_n$.
\end{lemma}

\begin{lemma}\label{le:BergmanEst03} {\rm (cf. \cite[Lemma 2.24]{Zhu2005})}
Suppose $\g >0, p>0,$ and $\alpha > -1.$ Then there exists a constant $C>0$ such that for any $f \in \mathcal{H} (\mathbb{B}_n),$
\be
|f(z) |^p \le \frac{C}{v_{\alpha}(D(z,\g))} \int_{D(z,\g)} | f(w)|^p d v_{\alpha} (w), \quad \forall z \in \mathbb{B}_n.
\ee
\end{lemma}

\subsection{Homogeneous spaces}

Recall that a quasimetric on a set $X$ is a map $\rho$ from $X \times X$ to $[0, \8)$ such that
\begin{enumerate}[{\rm (1)}]

\item $\rho (x, y) =0$ if and only if $x =y;$

\item $\rho (x, y) = \rho (y, x);$

\item there exists a positive constant $K \ge 1$ such that
\beq\label{eq:quasi-triangularinequa}
\rho (x, y) \le K [ \rho (x, z) + \rho (z, y)],\quad \forall x,y,z \in X,
\eeq
(the quasi-triangular inequality).

\end{enumerate}
For any $x \in X$ and $r >0,$ the set $B^{\rho} (x,r) = \{ y \in X: \; \rho (x,y) < r \}$ is called a $\rho$-ball of center $x$ and radius $r.$

A space of homogeneous type is a topological space $X$ endowed with a quasi-metric $\rho$ and a Borel measure $\mu$ such that
\begin{enumerate}[{\rm $\bullet$}]

\item for each $x \in X,$ the balls $B(x,r)$ form a basis of open neighborhoods of $x$ and, also, $\mu (B(x,r))>0$ whenever $r >0;$

\item (doubling property) there exists a constant $A>0$ such that for each $x \in X$ and $r>0,$ one has
\beq\label{eq:DoublingInequ}
\mu (B(x,2r)) \le A \mu (B(x,r)).
\eeq

\end{enumerate}
$(X, \rho, \mu)$ is called a space of homogeneous type or simply a homogeneous space. We will usually abusively call $X$ a homogeneous space instead of $(X, \rho, \mu).$ We refer to \cite{CW1977} for details on harmonic analysis on homogeneous spaces.

We now turn to our attention to the unit ball $\mathbb{B}_n,$ in which a quasi-metric $\varrho$ is defined as
\be
\varrho (z, w) = \left \{\begin{split}
& \big | |z| - |w| \big | + \Big | 1 - \frac{1}{|z| |w|} \langle z, w \rangle \Big |,\quad \text{if}\; z, w \in \mathbb{B}_n \backslash \{0\},\\
& |z| + |w|,\quad \text{otherwise}.
\end{split} \right.
\ee
It can be shown that the constant $K=2$ in the quasi-triangular inequality \eqref{eq:quasi-triangularinequa} for $\varrho$ and, $(\mathbb{B}_n, \varrho, d v_{\alpha})$ is a homogeneous space (see e.g. \cite{Bekolle1981, Tch2008}).

The following are some basic properties of the pseudo-metric $\varrho$ which will be used later, whose proofs can be found in \cite[Section 2.2]{Tch2008}.

\begin{lemma}\label{le:Psedu-metric01} {\rm (cf. \cite[Proposition 2.9]{Tch2008})}
For every pseudoball $B^{\varrho} (z, r),$ if $r > 1 -|z|$ then
\be
Q \Big ( \frac{z}{|z|}, r \Big ) \subset B^{\varrho} (z, 2 r) \quad \text{and}\quad B^{\varrho} (z, r) \subset Q \Big ( \frac{z}{|z|}, 4 r \Big ).
\ee
Here, $Q (\zeta, r) = \{ z \in \mathbb{B}_n: \; | 1 - \langle \zeta, z\rangle | < r \}$ for every $\zeta \in \mathbb{S}_n$ and $r >0.$
\end{lemma}

\begin{lemma}\label{le:Psedu-metric02} {\rm (cf. \cite[Lemma 2.10]{Tch2008})}
Let $\alpha > -1.$ For $z \in\mathbb{B}_n \setminus\{0\}$ and $0<r<3,$
\be
v_{\alpha}\left(B^{\varrho}(z,r)\right)\approx r^{n+1} [ \mathrm{max}(r,1-|z|) ]^{\alpha},
\ee
where ``$\approx$" depends only on $\alpha$ and $n.$
\end{lemma}

\begin{lemma}\label{le:Psedu-metric03} {\rm (cf. \cite[Lemma 2.12]{Tch2008})}
For $z \in\mathbb{B}_n$ and $0<r_0<1,$ if $z_0 = (r_0, 0, \ldots,0)$ one has
\begin{enumerate}[$\bullet$]

\item $|1 - r_0 z_1| \ge \frac{1}{3} \varrho (z, z_0);$

\item $|z_1 - r_0| \le \varrho (z, z_0);$

\item $\sum^n_{j=2} |z_j|^2 \le 2 \varrho (z, z_0);$

\item $|1 - \langle z, z_0 \rangle | \le 1 - r^2_0 + \varrho (z, z_0).$

\end{enumerate}
\end{lemma}

\subsection{Local coordinates}

For $z \in \mathbb{B}_n$ and $\xi \in \mathbb{C}^n$ a unit vector, we denote by $\tau (z, \xi)$ the distance from $z$ to the boundary $\mathbb{S}_n$ along the complex line determined by $\xi.$ For each $z_0 \in \mathbb{B}_n$ there exists a special set of real coordinate basis $\{ v_1 (z_0), \tau_1 (z_0),  \ldots, v_n (z_0), \tau_n (z_0)\}$ in $\mathbb{R}^{2 n} = \mathbb{C}^n$ defined as follows, which we call $\tau$-{\it extremal}. The first vector
\be
v_1 (z_0) = \left \{\begin{split}
& \frac{z_0}{|z_0|},\quad \text{if}\quad z_0 \not= 0,\\
& \; \mathbf{1}, \quad \; \text{if}\quad z_0 =0,
\end{split}\right.
\ee
where $\mathbf{1} =(1, 0, \ldots, 0).$ $v_1 (z_0)$ is clearly the direction transversal to the boundary $\mathbb{S}_n,$ in the sense that the shortest distance from $z$ to $\mathbb{S}_n$ is attained in the complex line determined by $v_1(z_0).$ The vector $v_2 (z_0)$ is chosen among the vectors orthogonal to $v_1(z_0)$ in such a way that $\tau (z_0, v_2 (z_0))$ is maximal. The vector $v_3(z_0)$ is chosen among the vectors orthogonal to both $v_1 (z_0)$ and $v_2(z_0)$ such that $\tau (z_0, v_3 (z_0))$ is maximal. We repeat this process until we obtain an orthonormal basis $\{v_1 (z_0), \ldots, v_n (z_0) \}$ in $\mathbb{C}^n.$ Put
\be
\tau_j (z_0) = \mathrm{i} v_j (z_0),\quad j=1, \ldots, n.
\ee
Then $\{ v_1 (z_0), \tau_1 (z_0),  \ldots, v_n (z_0), \tau_n (z_0) \}$ is an orthonormal basis in $\mathbb{R}^{2 n}.$ In particular, one has
\be
v_i (\mathbf{0}) = e_{2i-1} \quad \text{and}\quad \tau_i (\mathbf{0})= e_{2i},\quad i =1, \ldots, n,
\ee
where $e_k$ ($k=1,2, \ldots, 2n-1, 2n$) are the standard basis for $\mathbb{R}^{2 n}.$

For $w \in \mathbb{B}_n$ if
\be
w - z_0 = \alpha_1 v_1 (z_0) + \beta_1 \tau_1 (z_0) +  \ldots + \alpha_n v_n (z_0) + \beta_n \tau_n (z_0),
\ee
we denote by $( \alpha_1, \beta_1, \ldots, \alpha_n , \beta_n )$ the real coordinates of $w$ with respect to this basis. Precisely, we define a mapping $\Theta: \mathbb{B}_n \times \mathbb{B}_n \rightarrow \mathbb{R}^{2n}$ such that if
\be
w-z = \alpha_1 v_1 (z) + \beta_1 \tau_1 (z) +  \ldots + \alpha_n v_n (z) + \beta_n \tau_n (z)
\ee
then $\Theta(z,w)=(\alpha_1, \beta_1, \cdots, \alpha_n, \beta_n).$ One can easily verify that this coordinate mapping $\Theta$ is a $C^\infty$ diffeomorphism.

For a multi-index $J=(j_1, j_2, j_3, \cdots, j_{2n})$ in $\mathbb{N} \cup \{ 0 \},$ let
\be
d(J)=j_1+j_2+\frac{j_3}{2}+\cdots+\frac{j_{2n}}{2}
\ee
and $|J|=j_1+\cdots+j_{2n}.$ For any $f\in C^{\infty}(\mathbb{B}_n)$ and $z \in \mathbb{B}_n,$ we define a differential operator
\be
D^J_z f ( w ) = \frac{\partial^{j_1+\cdots+j_{2n}}}{\partial\alpha_1^{j_1}\partial \beta_1^{j_2} \cdots \partial\alpha_n^{j_{2n-1}} \partial \beta_n^{j_{2n}}} f \left( z + \alpha_1 v_1 (z) + \beta_1 \tau_1 ( z ) + \cdots + \alpha_n v_n (z) + \beta_n \tau_n ( z ) \right),
\ee
whenever $\Theta(z,w)=(\alpha_1, \beta_1, \cdots, \alpha_n, \beta_n).$ Let $\Omega$ be a domain in $\mathbb{B}_n,$ we say $f \in \mathcal{C}^{N}(\Omega)$ if for each $z \in \Omega$ and $J$ with $|J| \leq N,$ $D^J_{z}f(w)$ exists and is continuous in a neighborhood of $z.$

\section{Maximal functions}\label{Max}

In order to give a constructive proof of the atomic decomposition for Bergman spaces $\mathcal{A}_{\alpha}^p$ ($0< p \le 1 $), we need to introduce some maximal functions. These maximal functions are variants of the ones used for Hardy spaces (see e.g. \cite{FS1972, GP2002, KL1995}).

\subsection{Non-tangential and tangential maximal functions}

Let $\delta>0$ and $z\in\mathbb{B}_n.$ The \textquoteleft{approach region}' $A_{\delta}(z)$ is defined by
\be
A_{\delta}(z)=\left\{w\in\mathbb{B}_n:\ \varrho(z,w)<\delta(1-|w|)\right\}.
\ee
For any $f \in \mathcal{H} ( \mathbb{B}_n),$ we define respectively the non-tangential maximal function
\be
f_{\delta}^{\star}(z)= \sup_{w\in A_{\delta}(z)}|f(w)|
\ee
and the tangential function
\be
f^{\star\star}_{M}(z)=\sup_{w\in\mathbb{B}_n}\left(\frac{1-|w|}{1-|w|+\varrho(z,w)}\right)^{M}|f(w)|,
\ee
where $M$ is a positive constant.

\begin{lemma}\label{le:non-tangent}
Let $0<p<\infty$ and $\alpha > -1.$ Then
\be
\|f^{\star}_{\delta}\|_{p,\alpha}\lesssim \|f\|_{p,\alpha},\quad \forall f \in \mathcal{A}_{\alpha}^p.
\ee
\end{lemma}

\begin{proof}
Fix $\g >0.$ By \cite[Lemma 5.23]{Zhu2005}, there exists $\sigma \in(0,1)$ (depending only on $\g$) such that
\be
D(w, \g)\subset Q(\zeta,r^2)
\ee
for all $0<r<1$ and $\zeta\in\mathbb{S}_n,$ where
\be
w=(1-\sigma r^2)\zeta,\ \  Q(\zeta,r^2)=\{z\in\mathbb{B}_n:\ |1-\langle\zeta,z\rangle|<r^2\}.
\ee
Note that $1-|w|<r^2,$ by Lemma \ref{le:Psedu-metric01} we have
\be
Q(\zeta,r^2)=Q \Big ( \frac{w}{|w|},r^2 \Big ) \subset B^{\varrho} ( w, 4r^2 ) = B^{\varrho} \Big ( w,\frac{4}{\sigma}(1-|w|) \Big ).
\ee
Thus, for all $1-\sigma<|w|<1$ we obtain
\be
D(w, \g)\subset B^{\varrho} \Big ( w,\frac{4}{\sigma}(1-|w|) \Big ).
\ee
Let $z \in \mathbb{B}_n,$ $w \in A_{\delta}(z),$ and $1-\sigma<|w|<1.$ Then
\be\begin{split}
1-|z| = & 1-|z|-(1-|w|)+(1-|w|)\\
\leq & \varrho(z,w)+(1-|w|) \leq (\delta+1)(1-|w|)
\end{split}\ee
and so,
\be
B^{\varrho} \Big ( w,\frac{4}{\sigma}(1-|w|) \Big ) \subset B^{\varrho} \Big ( z, 2 \big [ \delta + \frac{4}{\sigma} \big ] (1-|w|) \Big ).
\ee
Moreover, by Lemma \ref{le:Psedu-metric02}
\be
v_{\alpha} \left ( B^{\varrho} \Big (z, 2 \big [ \delta+\frac{4}{\sigma} \big ] (1-|w|) \Big ) \right ) \approx \left ( \big [ \delta + \frac{4}{\sigma} \big ] (1-|w|) \right )^{n+1+\alpha}.
\ee

Now, by Lemmas \ref{le:BergmanEst01} and \ref{le:BergmanEst03} we have, for all $w \in A_{\delta}(z)$ and $1-\sigma<|w|<1,$
\be\begin{split}
|f(w)|^{\frac{p}{2}}&\leq\frac{C_{p,\alpha, \g}}{(1-|w|^2)^{n+1+\alpha}}\int_{D(w,\g)}|f(z)|^{\frac{p}{2}}dv_{\alpha}(z)\\
& \le \frac{C_{p,\alpha, \g}}{(1-|w|^2)^{n+1+\alpha}} \int_{B^{\varrho} \big ( w,\frac{4}{\sigma}(1-|w|) \big )} |f(z)|^{\frac{p}{2}} d v_{\alpha}(z)\\
& \lesssim \frac{C_{p,\alpha, \g} [ \delta+\frac{4}{\sigma} ]^{n+1+\alpha}}{ v_{\alpha} (B^{\varrho}(z, 2 [ \delta+\frac{4}{\sigma}] (1-|w|)))}
\int_{B^{\varrho} \big ( z, 2 (\delta+\frac{4}{\sigma})(1-|w|) \big )}|f(z)|^{\frac{p}{2}}dv_{\alpha}(z)\\
& \lesssim C_{p,\alpha, \g, \delta, \sigma} \mathcal{M}(|f|^{\frac{p}{2}})(z),
\end{split}\ee
where $\mathcal{M}$ is the central Hardy-Littlewood maximal operator defined by
\be
\mathcal{M}(f)(z)=\sup_{r>0}\frac{1}{v_{\alpha}(B^{\varrho}(z,r))}\int_{B^{\varrho}(z,r)}|f(w)|dv_{\alpha}(w).
\ee
By Lemma \ref{le:BergmanEst03} again, it follows that
\be\begin{split}
f_{\delta}^{\star}(z) & \le \left( \sup_{w \in \mathcal{A}_{\delta}(z), \; |w|>1-\sigma } |f(w)|^{\frac{p}{2}} \right)^{\frac{2}{p}}+\sup_{|w|\leq1-\sigma}|f(w)|\\
& \leq C_{p,\alpha,\g,\delta,\sigma} \big ( [ \mathcal{M}(|f|^{\frac{p}{2}}) (z) ]^{\frac{2}{p}} + \|f\|_{p,\alpha} \big ),
\end{split}\ee
and hence,
\be\begin{split}
\int_{\mathbb{B}_n}|f_{\delta}^{\star}(z)|^{p}d v_{\alpha} (z) & \lesssim \int_{\mathbb{B}_n } [(\mathcal{M}( |f|^{\frac{p}{2}}) (z) ]^{2} d v_{\alpha} + \|f\|_{p,\alpha}^p \lesssim\|f\|_{p,\alpha}^p,
\end{split}\ee
the last inequality is obtained by the fact that $\mathcal{M}$ is a bounded operator on $L^{q}(\mathbb{B}_n, dv_{\alpha})$ for any $1<q\le\infty$ (see \cite[Theorem 3.5]{CW1977}). This completes the proof.
\end{proof}

\begin{lemma}\label{le:tangent}
Let $0<p<\infty$ and $\alpha > -1.$ If $M$ is a constant such that $Mp > n + 1 + \alpha,$ then
\be
\|f^{\star\star}_{M}\|_{p,\alpha} \lesssim \|f\|_{p,\alpha}, \quad \forall f \in \mathcal{A}_{\alpha}^p.
\ee
\end{lemma}

\begin{proof}
By Lemma \ref{le:Psedu-metric02}, there exist two constants $c_{\alpha}$ and $C_{\alpha}$ such that
\be
c_{\alpha} r^{n+1}\mathrm{max}\{r,1-|z|\}^{\alpha} \le v_{\alpha}(B^{\varrho}(z,r)) \le C_{\alpha} r^{n+1}\mathrm{max}\{r,1-|z|\}^{\alpha}
\ee
for all $z \in \mathbb{B}_n$ and $0<r<3.$ Let $0<\beta<1$ and $N \in \mathbb{N}.$ We define the following two sets:
\be
E_{\beta} = \{ z \in \mathbb{B}_n: \; f^{\star}(z) \triangleq f^{\star}_1 ( z) > \beta \}
\ee
and
\be
E_{\beta, N}^{\star}=\left\{z\in\mathbb{B}_n:\; \mathcal{M}(\chi_{E_{\beta}})(z)> \frac{c_{\alpha}}{ 2 C_{\alpha} (N+1)^{n+1+\alpha}}\right\}
\ee
where $\chi_E$ is the characteristic function on $E.$ 
Since $\mathcal{M}$ is bounded form $L^1(\mathbb{B}_n,dv_{\alpha})$ to $L^{1,\infty}(\mathbb{B}_n,dv_{\alpha})$ (see \cite[Theorem 3.5]{CW1977}), we have
\be
v_{\alpha}(E_{\beta, N}^{\star})\lesssim (N+1)^{n+1+\alpha} v_{\alpha}(E_{\beta}).
\ee

On the other hand, if $z$ is not in $E_{\beta, N}^{\star},$ we claim that $f^{\star}_{N}(z)\leq\beta.$ In fact, for any $w \in A_{N}(z),$ the set $B^{\varrho}(w,1-|w|)$ can not be contained in $E_{\beta}.$ Otherwise,
\be\begin{split}
\mathcal{M_{\rho}}(\chi_{E_{\beta}})(z)& \geq\frac{\int_{B^{\varrho}(z,(N+1)(1-|w|))}\chi_{E_{\beta}}dv_{\alpha}}{v_{\alpha}
\big(B^{\varrho}(z, (N+1)(1-|w|))\big)}\\
&\geq\frac{\int_{B^{\varrho}(w,(1-|w|))}\chi_{E_{\beta}}dv_{\alpha}}{v_{\alpha}
\big(B^{\varrho}(z, (N+1)(1-|w|))\big)}\\
&\geq\frac{v_{\alpha}(B^{\varrho}(w,1-|w|))}{v_{\alpha}\big(B^{\varrho}(z, (N+1)(1-|w|))\big)}\\
&\geq\frac{c_{\alpha}(1-|w|)^{n+1+\alpha}}{C_{\alpha}\big( (N+1)(1-|w|)\big)^{n+1+\alpha}}\\
&>\frac{c_{\alpha}}{2C_{\alpha} (N+1)^{n+1+\alpha}},
\end{split}\ee
since  $1-|z|<(N+1)(1-|w|).$ This implies $z\in E^{\star}_{\beta, N}$ which contradicts our assumption. Hence, there exists $\xi\in B^{\varrho}(w,1-|w|)$ such that $f^{\star}(\xi)\leq\beta$ and so $|f(w)|\leq\beta.$ Thus $f^{\star}_{N}(z)\leq\beta.$ Then we have,
\be\begin{split}
\int_{\mathbb{B}_n}|f^{\star}_{N}(z)|^{p}dv_{\alpha}(z)&=\int_{\mathbb{B}_n}|f^{\star}_{N}(z)|^{p}dv_{\alpha}(z)\\
& = p \int_{0}^{1}\beta^{p-1}v_{\alpha}\{z\in\mathbb{B}_n :\; |f^{\star}_{N}(z)|>\beta\}d\beta\\
& \le p \int_{0}^{1}\beta^{p-1}v_{\alpha}(E_{\beta, N}^{\star})d\beta\\
& \lesssim  (N+1)^{n+1+\alpha}p\int_{0}^{1}\beta^{p-1}v_{\alpha}(E_{\beta})d\beta\\
& \lesssim  (N+1)^{n+1+\alpha} \int_{\mathbb{B}_n} |f^{\star}(z)|^p d v_{\alpha}(z).
\end{split}\ee

Now, if $\varrho(z,w)<1-|w|$ then
\be
\left(\frac{1-|w|}{1-|w|+\varrho(z,w)}\right)^{M}|f(w)|\leq f^{\star}(z),
\ee
and if $2^{k}(1-|w|)\leq\varrho(z,w)<2^{k+1}(1-|w|)$ for $k=0,1,2,\cdots,$
\be
\left(\frac{1-|w|}{1-|w|+\varrho(z,w)}\right)^{M}|f(w)|< \frac{1}{2^{kM}}f^{\star}_{2^{k+1}}(z).
\ee
Thus, by Lemma \ref{le:non-tangent},
\be\begin{split}
\int_{\mathbb{B}_n}|f^{\star\star}_{M}(z)|^pdv_{\alpha}(z)
& \leq \sum_{k=0}^{\infty} \int_{\mathbb{B}_n} \left | \frac{1}{2^{kM}}f^{\star}_{2^{k+1}}(z) \right |^p d v_{\alpha}(z)\\
& \leq \sum_{k=0}^{\infty}\frac{1}{2^{kMp}}\int_{\mathbb{B}_n}|f^{\star}_{2^{k+1}}(z)|^pdv_{\alpha}(z)\\
& \lesssim \sum_{k=0}^{\infty}  2^{k(n+1+\alpha-Mp)} \int_{\mathbb{B}_n} |f^{\star}(z)|^p dv_{\alpha}(z)\\
& \lesssim  \|f\|_{p,\alpha}^{p}\sum_{k=0}^{\infty}2^{k(n+1+\alpha-Mp)}\\
&\lesssim\|f\|_{p,\alpha}^{p}.
\end{split}\ee
The proof of Lemma \ref{le:tangent} is finished.
\end{proof}

\subsection{Grand maximal functions}

We next introduce the definition of so-called grand maximal functions, which will play an important role in the construction of atomic decomposition. To this end, we need to define a space of smooth bump functions. Let $\delta >0$ and let $L \ge 0$ be an integer. Given $z \in \mathbb{B}_n,$ we denote by $\mathcal{G}^{L}_{\delta}(z)$ the space of smooth bump functions at $z$ for $\delta$ and $L,$ that consists of all functions $g\in C^{\infty}(\mathbb{B}_n)$ for which there exist $z_{0}\in\mathbb{B}_n$ and $r_0 >0$ such that
\be
\mathrm{supp} g \subset B^{\varrho}(z_{0},r_{0}),\quad \varrho(z,z_{0})<\delta r_{0}, \quad \text{and}\; \|g\|_{L,z_{0},r_{0}}\leq 1,
\ee
where
\be
\|g\|_{L,z_{0},r_{0}}=v_{\alpha}(B^{\varrho}(z_0,r_0)) \sup_{|J|\le L}r_{0}^{d(J)} \|D^{J}_{z_0}g\|_{L^{\infty}(B^{\varrho}(z_0,r_0))}.
\ee

The grand maximal function on $\mathbb{B}_n$ is defined as
\be
\mathcal{K}_{\delta,L}(f)(z)=\sup_{g\in\mathcal{G}^{L}_{\delta}(z)}\left|\int_{\mathbb{B}_n}f(w)g(w)dv_{\alpha}(w)\right|.
\ee

\begin{lemma}\label{le:grand}
Let $0<p<\infty$ and $\delta >0.$ Let $L>M$ be an integer, where $M$ is a constant such that $Mp > n + 1 + \alpha.$ Then
\be
\mathcal{K}_{\delta,L}(f) (z) \lesssim f_{3+2\delta}^{\star}(z) + f^{\star\star}_{M}(z),\quad \forall f \in \mathcal{H} (\mathbb{B}_n),
\ee
for all $z\in\mathbb{B}_n.$ Consequently,
\be
\|\mathcal{K}_{\delta,L}(f)\|_{L^p (\mathbb{B}_n,dv_{\alpha})}\lesssim\|f\|_{p,\alpha},\; \forall f \in \mathcal{A}_{\alpha}^p.
\ee
\end{lemma}

\begin{proof}
Let $w=(x_1+\mathrm{i}y_1,\cdots,x_n+\mathrm{i}y_n).$ We define a differential operator $Y_{\el}$ of order $\el$ by
\be
Y_{\el}=\sum_{\mbox{\tiny$\begin{array}{c}k_1 + \cdots + k_n + \\ m_1 + \cdots + m_n = \el \end{array}$}}C_{k_1,\cdots,k_n,m_1,\cdots,m_n} ( w )
\frac{\partial^{k_1 + \cdots + k_n + m_1 + \cdots + m_n}}{\partial x_1^{k_{1}}\partial y_{1}^{m_1}\cdots \partial x_{n}^{k_{n}}\partial y_{n}^{m_n}}
\ee
with smooth coefficients $C_{k_1,\cdots, k_n, m_1, \cdots, m_n}.$ The symbol $Y_{\el}$ will denote different operators in different contexts. Then,
\beq\label{eq:deviative-inequality}
\|Y_{\el}g\|_{L^{\infty}(B^{\varrho}(z_0,r_0))}\lesssim\sum_{|J|=\el}\frac{1}{r_{0}^{d(J)}v_{\alpha}(B^{\varrho}(z_{0},r_{0}))}, \ \ \ \forall \el \leq L,
\eeq
for all $g\in\mathcal{G}^{L}_{\delta}(z).$ This can be obtained by the chain rule and change of coordinates. For example, let
\be
w-z_0 = \alpha_1 v_1 ( z_0 ) + \beta_1 \tau_1 ( z_0 ) + \cdots + \alpha_n v_n ( z_0 ) + \beta_n \tau_n ( z_0 ),
\ee
by orthornormality, we have
\be
\alpha_j =(w-z_{0})\bullet v_j ( z_0 )\quad \text{and}\quad \beta_j = (w-z_{0})\bullet \tau_j (z_0),\quad j =1, \cdots, n,
\ee
where $\textquoteleft{\bullet}$' is dot-productor on $\mathbb{R}^{2n}$. Using the chain rule, we have
\be
\frac{\partial g}{\partial x_{1}}(w)=\frac{\partial g}{\partial \alpha_{1}}(w)\frac{\partial \alpha_{1}}{\partial x_{1}}+\cdots+\frac{\partial g}{\partial \beta_{n}}(w)\frac{\partial \beta_{n}}{\partial x_{1}}.
\ee
Hence,
\be\begin{split}
\left\|\frac{\partial g}{\partial x_{1}}\right\|_{L^{\infty}(B^{\varrho}(z_0,r_0))} \lesssim \sum_{|J|=1}\left\|D^{J}_{z_{0}}g\right\|_{L^{\infty}(B^{\varrho}(z_0,r_0))} \leq \sum_{|J|=1}\frac{1}{r_{0}^{d(J)}v_{\alpha}(B^{\varrho}(z_{0},r_{0}))}.
\end{split}\ee
Higher derivatives can be estimated similarly and (\ref{eq:deviative-inequality}) is obtained.

Notice that for $0<r<1$ and $\xi\in\mathbb{S}_n$, we shall extend $g(r\xi)$ form $\mathbb{S}_n$ to a function $G_{r}\in C^{\infty}(\overline{\mathbb{B}}_n)$. In fact, we can let
\be
G_{r}(w) = \left \{\begin{split}
& g(rw)g_{1}(1-|w|),\quad \text{if}\ 1-2r_0\leq|w|\leq1,\\
& \quad\quad\quad0 \quad\quad\quad\quad\quad \ \text{if}\ |w|<1-2r_0,
\end{split} \right.
\ee
where $g_{1}\in C^{\infty}_{0}([0,2r_0])$, $g_{1}=1$ if $0\leq t\leq r_{0}$ and $|g_{1}^{(j)}|\leq c_{j}r_{0}^{-j}$ ($c_{j}$ depending only on $j$). By (\ref{eq:deviative-inequality}), we have
\be\begin{split}
|Y_{\el} G_{r} (w) |
& \lesssim \sum_{k=0}^{\el}\frac{1}{r_0^{k}}|(Y_{\el-k}g)(rw)| \lesssim \sum_{k=0}^{\el} \frac{1}{r_0^{k}} \sum_{|J|=\el-k} \frac{1}{r_{0}^{d(J)} v_{\alpha}(B^{\varrho} ( z_{0}, r_{0}))}\\
& \leq c_{\el} \sum_{k=0}^{\el}\frac{1}{r_0^{k}} \cdot \frac{1}{r_{0}^{\el- k} v_{\alpha} (B^{\varrho}(z_{0}, r_{0}))} \leq \frac{c_{\el}}{r_{0}^{\el} v_{\alpha}(B^{\varrho}(z_{0}, r_{0}))}.
\end{split}\ee

Now we estimate the integration
\be\begin{split}
&\int_{\mathbb{B}_n}f(w)g(w)dv_{\alpha}(w)\\
&=2n\int^{1}_0 r^{2n-1}(1-r^2)^{\alpha}
\left(\int_{\mathbb{S}_n}f(r\xi)g(r\xi)d\sigma(\xi)\right)dr\\
&=2n\int^{1}_0 r^{2n-1}(1-r^2)^{\alpha}
\left(\int_{\mathbb{S}_n}[f(r\xi)-f(r\xi-r_{0}r\xi)]g(r\xi)d\sigma(\xi)\right)dr\\
&\qquad+2n\int^{1}_0 r^{2n-1}(1-r^2)^{\alpha}
\left(\int_{\mathbb{S}_n}f(r\xi-r_{0}r\xi)g(r\xi)d\sigma(\xi)\right)dr\\
&=I_{1}+I_{2}.
\end{split}\ee
For $I_{2}$ we note that when $r\xi\in B^{\varrho}(z_0,r_0)$, it implies
\be
\varrho(r\xi,z) \le 2 [ \varrho(r\xi,z_{0})+\varrho(z_{0},z) ] \le (2+2\delta) r_{0},
\ee
moreover,
\be
\varrho(r\xi-r_{0}r\xi,z) \le \big||r\xi-r_{0}r \xi | - |r\xi|\big|+\varrho(r\xi,z) \le (3+2\delta)r_{0}.
\ee
On the other hand, we know that
\be
1-|r\xi-r_{0}r\xi|>r_{0}.
\ee
Hence, $r\xi-r_{0}r\xi\in A_{(3+2\delta)}(z)$. Therefore,
\be\begin{split}
|I_{2}|&=\left|2n\int^{1}_0 r^{2n-1}(1-r^2)^{\alpha}
\left(\int_{\mathbb{S}_n}f(r\xi-r_{0}r\xi)g(r\xi)d\sigma(\xi)\right)dr\right|\\
& \leq \frac{ f_{3+2\delta}^{\star}(z)}{v_{\alpha}(B^{\varrho}(z_0,r_0))}2n \int^{1}_0 r^{2n-1}(1-r^2)^{\alpha}
\left(\int_{\{ \xi \in \mathbb{S}_n: \varrho (r \xi, z_0)< r_0\}}1d\sigma(\xi)\right)dr\\
&\leq f_{3+2\delta}^{\star}(z).
\end{split}\ee

Next we estimate the term $I_{1}.$ Let $v_{\xi}$ denote the unit outward vector at $\xi\in\mathbb{S}_n$, then
\be\begin{split}
I_{1}&=2n\int^{1}_0 r^{2n-1}(1-r^2)^{\alpha}
\left(\int_{\mathbb{S}_n}[f(r\xi)-f(r\xi-r_{0}r\xi)]g(r\xi)d\sigma(\xi)\right)dr\\
&=2n\int^{1}_0 r^{2n-1}(1-r^2)^{\alpha}
\left(\int_{\mathbb{S}_n}\left[\int^{r_{0}}_{0}-\frac{d f}{dr_{1}}(r\xi-r_{1}r\xi)dr_{1}\right]g(r\xi)d\sigma(\xi)\right)dr\\
&=2n\int^{1}_0 r^{2n-1}(1-r^2)^{\alpha}\int^{r_{0}}_{0}r
\left(\int_{\mathbb{S}_n}\frac{d f}{dv_{\xi}}(r\xi-r_{1}r\xi)G_{r}(\xi)d\sigma(\xi)\right)dr_{1}dr\\
\end{split}\ee
By applying the Green's formula, we have
\be\begin{split}
I_{1}=&2n\int^{1}_0 r^{2n-1}(1-r^2)^{\alpha}\int^{r_{0}}_{0}r
\left(\int_{\mathbb{S}_n}f(r\xi-r_{1}r\xi)\frac{d G_{r}(\xi)}{dv_{\xi}}d\sigma(\xi)\right)dr_{1}dr\\
&+2n\int^{1}_0 r^{2n-1}(1-r^2)^{\alpha}\int^{r_{0}}_{0}r
\left(\int_{\mathbb{B}_n}\triangle f(rw-r_{1}rw)G_{r}(w)dv(w)\right)dr_{1}dr\\
&-2n\int^{1}_0 r^{2n-1}(1-r^2)^{\alpha}\int^{r_{0}}_{0}r
\left(\int_{\mathbb{B}_n} f(rw-r_{1}rw)\triangle G_{r}(w)dv(w)\right)dr_{1}dr\\
=&I_{11}+0-I_{12},
\end{split}\ee
since $f(rw-r_{1}rw)$ is holomorphic in $\mathbb{B}_n.$ Let $Y_{2}G_{r}(w)\triangleq \triangle G_{r}(w).$ We rewrite $I_{12}=I_{121}+I_{122},$ where
\be\begin{split}
I_{121}= 2n\int^{1}_0 \int^{r_{0}}_{0} r^{2n-1}(1-r^2)^{\alpha}r \int_{\mathbb{B}_n\setminus\mathbb{D}_{1-r_{0}}} f(r w-r_{1}r w)Y_{2}G_{r}(w)dv(w) dr_{1} d r\\
\end{split}\ee
and
\be\begin{split}
I_{122}=  2n\int^{1}_0 \int^{r_{0}}_{0} r^{2n-1}(1-r^2)^{\alpha}r \int_{\mathbb{D}_{1-r_{0}}} f(r w - r_{1} r w) Y_{2}G_{r}(w)dv(w)  d r_{1}d r.
\end{split}\ee
For the term $I_{122},$ when $w\in\mathbb{D}_{1-r_{0}}$ we know that
\be
\varrho(rw-r_{1}rw,z)<(3+2\delta)r_{0}<(3+2\delta)(1-|rw-r_{1}rw|),
\ee
then,
\be\begin{split}
|I_{122}|=
& 2n\int^{1}_0 r^{2n-1}(1-r^2)^{\alpha}\int^{r_{0}}_{0}r \int_{1-2r_{0}}^{1-r_{0}}2ns^{2n-1}\\
& \times \left(\int_{\mathbb{S}_n} f(rs\xi-r_{1}rs\xi)Y_{2}G_{r}(s\xi)d\sigma(\xi)\right)d s dr_{1}dr\\
\lesssim & \frac{f_{3+2\delta}^{\star}(z)r_{0}}{r_{0}^2 v_{\alpha}(B^{\varrho}(z_0,r_0))}2n \int^{1}_0 r^{2n-1}(1-r^2)^{\alpha}\\
&\times\left(\int_{1-2r_{0}}^{1-r_{0}}2ns^{2n-1}\int_{\{\xi\in\mathbb{S}_n:\varrho(rs\xi,z_{0})<r_{0}\}}1d\sigma(\xi)ds\right)dr\\
\lesssim&\frac{f_{3+2\delta}^{\star}(z)r_{0}}{r_{0}^2v_{\alpha}(B^{\varrho}(z_0,r_0))}2n\int^{1}_0 r^{2n-1}(1-r^2)^{\alpha}\\
&\times\left(\int_{1-2r_{0}}^{1-r_{0}}2ns^{2n-1}\int_{\{\xi\in\mathbb{S}_n:\varrho(r\xi,z_{0})<3r_{0}\}}1d\sigma(\xi)ds\right)dr\\
\lesssim & \frac{f_{3+2\delta}^{\star}(z)r_{0}^{2}}{r_{0}^2v_{\alpha}(B^{\varrho}(z_0,r_0))}v_{\alpha}(B^{\varrho}(z_0,3r_0))\\
\lesssim & f_{3+2\delta}^{\star}(z).
\end{split}\ee

Therefore, we have
\be\begin{split}
\Big | & \int_{\mathbb{B}_n}f(w)g(w)dv_{\alpha}(w) \Big |\\
& \lesssim f_{3+2\delta}^{\star}(z)\\
& \quad + \left | \int^{1}_0 r^{2n-1}(1-r^2)^{\alpha}\int^{r_{0}}_{0}r
\left(\int_{\mathbb{S}_n}f(r\xi-r_{1}r\xi)\frac{d G_{r}(\xi)}{dv_{\xi}}d\sigma(\xi)\right)dr_{1}dr\right|\\
& \quad +  \Bigg | \int^{1}_0 r^{2n-1}(1-r^2)^{\alpha} \\
& \quad \quad \times \int^{r_{0}}_{0} \left ( \int_{\mathbb{B}_n\setminus\mathbb{D}_{1-r_{0}}} f(rw-r_{1}rw)Y_{2}G_{r}(w)dv(w)\right)dr_{1} d r \Bigg |\\
& = f_{3+2\delta}^{\star}(z)\\
& \quad + \left| \int^{1}_0 r^{2n-1}(1-r^2)^{\alpha}\int^{r_{0}}_{0}
\int_{\mathbb{S}_n}f\big((1-r_{1})r\xi\big)Y_{1}G_{r}(\xi)d\sigma(\xi)dr_{1} d r \right|\\
& \quad +  \Bigg |\int^{1}_0 r^{2n-1}(1-r^2)^{\alpha}\\
& \quad \quad \times \int^{r_{0}}_{0} \int_{1-r_{0}}^1 \int_{\mathbb{S}_n} f \big ( (1-r_{1})r s_{1}\xi\big)Y_{2}G_{r}(s_{1}\xi)d\sigma(\xi)d s_{1}dr_{1}  d r \Bigg |
\end{split}\ee
where, for simplicity, we let
\be
Y_{1}G_{r}(\xi)\triangleq\frac{d G_{r}(\xi)}{dv_{\xi}}
\ee
and replace the term
\be
2n s_{1}^{2n-1}Y_{2}G_{r}(s_{1}\xi)
\ee
by $Y_{2}G_{r}(s_{1}\xi)$.

By repeating the method used in the estimations of $I_{1}$ and $I_{2},$ we shall have the following estimate
\be\begin{split}
\Big | & \int_{\mathbb{B}_n}f(w)g(w)dv_{\alpha}(w) \Big |\\
& \lesssim f_{3+2\delta}^{\star}(z) + \sum_{0\leq k\leq\el} \int^{1}_0 r^{2n-1}(1-r^2)^{\alpha}
 \int^{r_{0}}_{0}\int_{r_{1}}^{r_{0}}\cdots\int^{r_{0}}_{r_{\el-1}}\int_{1-r_{0}}^{1}\cdots\int^{1}_{1-r_{0}}\\
&\quad \times \left | \int_{\mathbb{S}_n}f\big((1-r_{\el})r s_{1}\cdots
s_{k}\xi\big)Y_{\el+k}G_{r}(s_{1}\cdots s_{k}\xi)d\sigma(\xi)\right|d s_{k}\cdots d s_{1}d r_{\el}\cdots d r_{1}dr,
\end{split}\ee
which also can be obtained by mathematical induction and we omit the details. Notice that
\be\begin{split}
\left | f \big ((1-r_{\el})r s_{1} \cdots s_{k}\xi\big) \right |
\leq & f^{\star\star}_{M}(z) \left(1+\frac{\varrho\big((1-r_{\el})r s_{1}\cdots
s_{k}\xi,z\big)}{1-(1-r_{\el})r s_{1}\cdots
s_{k}}\right)^{M}\\
\leq & f^{\star\star}_{M}(z)\left(1+\frac{(3+2\delta)r_{0}}{r_{\el}}\right)^{M}\\
\leq & f^{\star\star}_{M}(z)(4+2\delta)^{M}\left(\frac{r_{0}}{r_{\el}}\right)^{M}
\end{split}\ee
and
\be\begin{split}
\varrho(r\xi,z_{0})\leq&|r-r s_{1}\cdots s_{k}|+\varrho(r s_{1}\cdots
s_{k}\xi,z_{0})\\
\leq&|r-r s_{1}|+\sum_{i=1}^{k-1}|r s_{1}\cdots s_{i}-r s_{1}\cdots s_{i}s_{i+1}|+\varrho(r s_{1}\cdots
s_{k}\xi,z_{0})\\
\leq&k r_{0}+r_{0}\leq(k+1)r_{0}.
\end{split}\ee
Now we fix $\el=L>M,$ then
\be\begin{split}
\Big | & \int_{\mathbb{B}_n}f(w)g(w)dv_{\alpha}(w) \Big |\\
\lesssim & f_{3+2\delta}^{\star}(z) + f^{\star\star}_{M}(z) \sum_{0\leq k\leq \el} \int^{1}_0 r^{2n-1}(1-r^2)^{\alpha}
\int^{r_{0}}_{0}\int_{r_{1}}^{r_{0}}\cdots\int^{r_{0}}_{r_{\el-1}}\int_{1-r_{0}}^{1}\cdots\int^{1}_{1-r_{0}}\\
& \times \int_{\{\xi \in \mathbb{S}_n: \varrho(r s_{1}\cdots
s_{k}\xi,z_0) <r_0\}}\left(\frac{r_{0}}{r_{\el}}\right)^{M}\frac{1}{r_{0}^{\el+k} v_{\alpha}(B^{\varrho}(z_{0}, r_{0}))}d\sigma(\xi)
d s_{k}\cdots d s_{1}d r_{\el}\cdots d r_{1}dr\\
\lesssim & f_{3+2\delta}^{\star}(z) + f^{\star\star}_{M}(z)\sum_{0\leq k \leq \el} \int^{1}_0 r^{2n-1}(1-r^2)^{\alpha}
\int^{r_{0}}_{0}\int_{r_{1}}^{r_{0}}\cdots\int^{r_{0}}_{r_{\el-1}}\int_{1-r_{0}}^{1}\cdots\int^{1}_{1-r_{0}}\\
& \times \int_{\{\xi \in \mathbb{S}_n: \varrho(r\xi,z_0) <(k+1)r_0\}}\left(\frac{r_{0}}{r_{\el}}\right)^{M}\frac{1}{r_{0}^{\el+k} v_{\alpha}(B^{\varrho}(z_{0}, r_{0}))}d\sigma(\xi)d s_{k}\cdots d s_{1}d r_{\el}\cdots d r_{1}dr\\
=& f_{3+2\delta}^{\star}(z) + \sum_{0\leq k\leq\el}\frac{f^{\star\star}_{M}(z) r_{0}^{k+M}}{r_{0}^{\el+k} v_{\alpha}(B^{\varrho}(z_{0}, r_{0}))} \int^{1}_0 r^{2n-1}(1-r^2)^{\alpha}
\int^{r_{0}}_{0}\int_{r_{1}}^{r_{0}}\cdots\int^{r_{0}}_{r_{\el-1}}\\
&\times\int_{\{\xi \in \mathbb{S}_n: \varrho(r\xi,z_0) <(k+1)r_0\}}\left(\frac{1}{r_{\el}}\right)^{M}d\sigma(\xi)d r_{\el}\cdots d r_{1}dr\\
\lesssim&f_{3+2\delta}^{\star}(z)+\sum_{0\leq k\leq\el}\frac{f^{\star\star}_{M}(z)r_{0}^{k+M}}{r_{0}^{\el+k} v_{\alpha}(B^{\varrho}(z_{0}, r_{0}))}
r_{0}^{\el-M}v_{\alpha}\big(B^{\varrho}(z_0,(k+1)r_0)\big)\\
\lesssim&f_{3+2\delta}^{\star}(z) + f^{\star\star}_{M}(z).
\end{split}\ee
Therefore, we have obtained that
\be
\left |\int_{\mathbb{B}_n}f(w)g(w)dv_{\alpha}(w) \right | \lesssim f^{\star}_{3+2\delta}(z) + f^{\star\star}_{M}(z)
\ee
for all $g\in\mathcal{G}^{L}_{\delta}(z)$ and $z\in\mathbb{B}_n.$ Hence, by Lemmas \ref{le:non-tangent} and \ref{le:tangent} we obtain the required result.
\end{proof}

\section{Statement of main results}\label{result}

In order to state the real-variable atomic decomposition of Bergman spaces for $0 < p \le 1,$ we will introduce the notion of $(p, \8)_{\alpha}$-atom. To this end, we need more notations.

Let $0<p \le 1$ and $\alpha > -1.$ Set
\be
N_{p, \alpha} = \max \bigg \{ \Big [ 2(n+1) \big ( \frac{1}{p}-1 \big ) \Big ],\; \Big [ 2(n+1+\alpha) \big ( \frac{1}{p}-1 \big ) \Big ] \bigg \}+1,
\ee
where $[x]$ denotes the greatest integer less than $x.$ Let $z_0 \in \mathbb{B}_n$ and $r_0 >0.$ For any $\phi\in\mathcal{C}^\infty(B^{\varrho}(z_{0},r_{0}))$, we define the quantity
\be
\|\phi\|_{\mathcal{S}_{N}(B^{\varrho}(z_{0},r_{0}))}: = \sum_{|J|= N}r_0^{d(J)}\left\|D^{J}_{z_0}\phi\right\|_{L^{\infty}(B^{\varrho}(z_{0},r_{0}))}.
\ee

\begin{definition}\label{df:p-atom}
Let $0<p \le 1$ and $\alpha > -1.$ Let $N \ge N_{p, \alpha}$ be an integer. A measurable function $a$ on $\mathbb{B}_n$ is a $(p, \8, N)_{\alpha}$-atom if there exist $z_0 \in \mathbb{B}^n$ and $r_0 >0$ such that
\begin{enumerate}[{\rm (1)}]
\item $a$ is supported in $B^{\varrho}(z_{0},r_{0});$

\item $|a(z)| \le v_{\alpha} (B^{\varrho}(z_{0},r_{0}))^{-\frac{1}{p}};$

\item $\int_{\mathbb{B}_n}a(z)\,dv_{\alpha}(z) = 0;$

\item for all $\phi\in\mathcal{C}^\infty(B^{\varrho}(z_{0},r_{0})),$
\be
\left|\int_{\mathbb{B}^n}a(z)\phi(z)dv_{\alpha}(z)\right|\le
\|\phi\|_{\mathcal{S}_N (B^{\varrho}(z_{0},r_{0}))}
v_{\alpha}(B^{\varrho}(z_{0},r_{0}))^{1-\frac{1}{p}}.
\ee

\end{enumerate}
Any bounded function $a$ with $\|a\|_{L^\infty}\leq1$ is also considered to be a $(p, \8, N)_{\alpha}$-atom.

We regard a $(p, \8, N_{p, \alpha})_{\alpha}$-atom as a $(p, \8)_{\alpha}$-atom.
\end{definition}

\begin{remark}\label{rk:atomBoundedfunct}\rm
\begin{enumerate}[{\rm (i)}]

\item As in the case of Hardy spaces on the complex ball (see the definition of $p$-atoms in \cite[Section 4]{GL1978}), it seems necessary to consider all bounded functions $a$ with $\|a\|_{L^{\8}} \le 1$ as atoms in the complex ball (see Section \ref{AtomConstruction} below).

\item We remark that the condition (4) replaces the classical higher moment condition and is similar to the one found in \cite{GP2002, KL1995}.


\end{enumerate}
\end{remark}

These atoms satisfy the following useful estimates.

\begin{theorem}\label{th:p-atomBergman}
Let $0<p \le 1$ and $\alpha > -1.$ Let $N \ge N_{p, \alpha}$ be a integer.
\begin{enumerate}[{\rm (1)}]

\item For any $(p, \8, N)_{\alpha}$-atom $a,$ $\|a \|_{p, \alpha} \le 1.$

\item There is a constant $C>0$ depending only on $p, n,\alpha$ and $N$ such that
\be
\left \|P_{\alpha}(a) \right \|_{p, \alpha} \leq C
\ee
for any $(p, \8, N)_{\alpha}$-atom $a.$

\item If $\{a_k\}$ is a sequence of $(p, \8)_{\alpha}$-atoms, then for any sequence of complex numbers $\{\lambda_k\}$ with $\sum_k |\lambda_k|^p < \8,$ the series $\sum_k \lambda_k P_{\alpha} (a_k)$ converges unconditionally in $\mathcal{A}^p_{\alpha}$ such that
\be
f : = \sum_k \lambda_k P_{\alpha} (a_k) \in \mathcal{A}^p_{\alpha} \quad \text{and} \quad \| f \|^p_{p , \alpha} \lesssim \sum_k |\lambda_k|^p.
\ee

\end{enumerate}
\end{theorem}

The following theorem shows that the converse statement of Theorem \ref{th:p-atomBergman} (3) holds true as well. This is the main result of this paper.

\begin{theorem}\label{th:atomic-deco}
Let $0<p\le1$ and $\alpha > -1.$ Let $N \ge N_{p, \alpha}$ be a integer. If $f \in \mathcal{A}^p_{\alpha},$ then there exist a scalar sequence $\{\lambda_j\}$ in $\mathbb{C}$ with $\sum_j |\lambda_j |^p < \8,$ and a sequence of $(p,\infty, N)_{\alpha}$-atoms $\{a_j\}$ such that
\begin{enumerate}[{\rm (1)}]


\item $f = \sum_j \lambda_j a_j$ in the sense of distributions;

\item $f = \sum_j \lambda_j a_j,$ where the series $\sum_j \lambda_j a_j$ converges unconditionally in $L^p_{\alpha} (\mathbb{B}_n);$

\item $f = \sum_j \lambda_j P_{\alpha}(a_j),$ where the series $\sum_j \lambda_j P_{\alpha} (a_j)$ converges unconditionally in $\A^p_{\alpha};$

\item $\sum_j |\lambda_j |^p \lesssim \| f \|^p_{p , \alpha}.$

\end{enumerate}

Consequently, for any $f \in \mathcal{A}^p_{\alpha}$ one has
\be
\| f \|^p_{p , \alpha} \approx \inf \bigg \{ \sum_j |\lambda_j |^p:\; f = \sum_j \lambda_j P_{\alpha} (a_j) \bigg \},
\ee
where the infimum is taken over all decompositions of $f$ described above.
\end{theorem}

In the sequel we will prove Theorem \ref{th:p-atomBergman}, while Theorem \ref{th:atomic-deco} will be proved in Section \ref{atomicdecomp} below.

The following lemma is crucial for the proof of Theorem \ref{th:p-atomBergman}.

\begin{lemma}\label{le:Kernel}
Let $z_{0}\in\mathbb{B}_n$ and $J=(j_1, j_2, j_3, \cdots, j_{2n})$ be a multi-index such that $|J|=N.$ If $\varrho(w,z_0) < \frac{1}{4} \varrho(z,z_0),$ then
\be
\left|D^J_{z_{0}}\big(K^{\alpha}(z,\cdot)\big)(w)\right|\le\frac{C_{N,n,\alpha}}{\varrho(z,w)^{d(J)}v_{\alpha}(B^{\varrho}(w,\varrho(z,w)))},
\ee
where $C_{N,n,\alpha}$ depending only on $N,n,\alpha$.
\end{lemma}

\begin{proof}
Assuming that
\be
w-z_0 = \alpha_1 v_1 ( z_0 ) + \beta_1 \tau_1 ( z_0 ) + \cdots + \alpha_n v_n ( z_0 ) + \beta_n \tau_n ( z_0 ),
\ee
we have
\be\begin{split}
D^J_{z_{0}} & \big(K^{\alpha}(z,\cdot)\big)(w)\\
=& C_{n,\alpha,J}\frac{\langle z, v_1 (z_0) \rangle^{j_1}
\langle z, \tau_1 (z_0) \rangle^{j_2}
\cdots \langle z, v_n (z_0) \rangle^{j_{2n-1}} \langle z, \tau_n ( z_0) \rangle^{j_{2n}}
}{\big(1-\langle z,w\rangle\big)^{n+1+\alpha+j_1+j_2+j_3+\cdots+j_{2n}}}.
\end{split}\ee
We are going to prove that
\be\begin{split}
\left|\frac{\langle z, v_i ( z_0) \rangle^{j}}{\big(1-\langle z,w\rangle\big)^{j}}\right|\leq\frac{C_N}{\varrho(z,w)^{\frac{j}{2}}}
\end{split}\ee
and
\be\begin{split}
\left|\frac{\langle z, \tau_i ( z_0) \rangle^{j}}{\big(1-\langle z,w\rangle\big)^{j}}\right|\leq\frac{C_N}{\varrho(z,w)^{\frac{j}{2}}}
\end{split}\ee
for all $2\leq i \leq n$ and $0\leq j\leq N.$ These two type inequalities can be proved in a similar way. So, we only prove the second one.

To this end, notice that $\varrho(w,z_0)<\frac{1}{4}\varrho(z,z_0).$ Then
\begin{enumerate}[$\bullet$]

\item $\big \langle z, \tau_i ( z_0 ) \big \rangle = \big \langle z-z_{0}, \tau_i  ( z_0 ) \big\rangle;$

\item $2 [ \varrho(z,w)+\varrho(w,z_{0}) ] \geq\varrho(z,z_{0});$

\item $\varrho(z,w)\geq\frac{1}{2}\big(\varrho(z,z_{0})-2\varrho(w,z_{0})\big)>\frac{1}{4}\varrho(z,z_{0}).$

\end{enumerate}
By Lemma \ref{le:Psedu-metric03}, we have
\be
\left|1-\langle z,w\rangle\right|\geq\frac{1}{4}\varrho(z,w) \quad \text{and}\quad
|z-z_{0}| \leq \sqrt{\varrho(z,z_{0})^{2} + 2\varrho(z,z_{0})}.
\ee
Hence, if $\varrho(z,z_{0})\leq1$ then
\be\begin{split}
\left | \frac{\langle z, \tau_i ( z_0 ) \rangle^{j}}{\big(1-\langle z,w\rangle\big)^{j}} \right |
= & \left | \frac{\langle z-z_{0}, \tau_i ( z_0 ) \rangle^{j}}{\big(1-\langle z,w\rangle\big)^{j}} \right|
\leq \left | \frac{z-z_{0}}{1-\langle z,w\rangle} \right |^{j}\\
\lesssim & \frac{\varrho(z,z_{0})^{\frac{j}{2}}}{\varrho(z,w)^{j}} \lesssim\frac{1}{\varrho(z,w)^{\frac{j}{2}}}.
\end{split}\ee
If $\varrho(z,z_{0})>1,$ then
\be\begin{split}
\left | \frac{\langle z, \tau_i ( z_0 ) \rangle^{j}}{\big(1-\langle z,w\rangle\big)^{j}} \right |
= & \left | \frac{\langle z-z_{0}, \tau_i ( z_0 ) \rangle^{j}}{\big(1-\langle z,w\rangle\big)^{j}} \right |
\leq \left | \frac{z-z_{0}}{1-\langle z,w\rangle} \right |^{j}\\
\lesssim & \frac{\varrho(z,z_{0})^{j}}{\varrho(z,w)^{j}}
\leq C_N \frac{\varrho(z,z_{0})^{\frac{j}{2}}}{\varrho(z,w)^{\frac{j}{2}}}
\leq C_{N} \frac{1}{\varrho(z,w)^{\frac{j}{2}}},
\end{split}\ee
since $\varrho(z,z_{0}) \leq 3.$ Thus,
\be\begin{split}
|D^J_{z_{0}} & \big(K^{\alpha}(z,\cdot)\big)(w)|\\
\leq & C_{n,\alpha,J}\left|\frac{1}{\big(1-\langle z,w\rangle\big)^{n+1+\alpha+j_1+j_2}}\cdot
\frac{
\langle z, v_2 (z_0) \rangle^{j_3} \cdots
\langle z, \tau_n ( z_0)\rangle^{j_2n}
}{\big(1-\langle z,w\rangle\big)^{j_3+\cdots+j_{2n}}}\right|\\
\leq & \frac{C_{N,n,\alpha}}{\varrho(z,w)^{d(J)}\big|1-\langle z,w\rangle\big|^{n+1+\alpha}}.
\end{split}\ee

Now, if $1-|w|\leq \varrho(z,w),$ then by Lemma \ref{le:Psedu-metric02}, $v_{\alpha}(B^{\varrho}(w,\varrho(z,w)))\approx\varrho(z,w)^{n+1+\alpha}$ and
\be\begin{split}
|D^J_{z_{0}}\big(K^{\alpha}(z,\cdot)\big)(w)| & \lesssim \frac{C_{N,n,\alpha}}{\varrho(z,w)^{d(J)} \varrho(z,w)^{n+1+\alpha}}\\
& \lesssim \frac{C_{N,n,\alpha}}{\varrho(z,w)^{d(J)}v_{\alpha}(B^{\varrho}(w,\varrho(z,w)))}.
\end{split}\ee
If $1-|w|>\varrho(z,w),$ by Lemma \ref{le:Psedu-metric02} again we have $v_{\alpha}(B^{\varrho}(w,\varrho(z,w)))\approx\varrho(z,w)^{n+1}(1-|w|)^{\alpha}.$ Then
\be\begin{split}
|D^J_{z_{0}}\big(K^{\alpha}(z,\cdot)\big)(w)|& \leq \frac{C_{N,n,\alpha}}{\varrho(z,w)^{d(J)}(1-|w|)^{n+1+\alpha}}\\
& \leq \frac{C_{N,n,\alpha}}{\varrho(z,w)^{d(J)}\varrho(z,w)^{n+1}(1-|w|)^{\alpha}}\\
& \lesssim \frac{C_{N,n,\alpha}}{\varrho(z,w)^{d(J)}v_{\alpha}(B^{\varrho}(w,\varrho(z,w)))}.
\end{split}\ee
The proof is complete.
\end{proof}

Now we turn to prove Theorem \ref{th:p-atomBergman}.

\

{\it Proof of Theorem \ref{th:p-atomBergman}}.\; Let $0 < p \le 1$ and $\alpha > -1.$ Suppose $N \ge N_{p, \alpha}$ is an integer. The assertion (1) follows immediately from the conditions (1) and (2) of Definition \ref{df:p-atom}, while the assertion (2) implies the assertion (3). It remains to prove the assertion (2).

To this end, letting $a$ be a $(p, \8, N)_{\alpha}$-atom, we need to prove that
\be
\left\|P_{\alpha}(a)\right\|_{\alpha,p} \leq C,
\ee
where $C$ depends only on $p, n, \alpha $ and $N.$ If $\|a\|_{L^{\infty}}\leq1$, then
\be\begin{split}
\left\|P_{\alpha}(a)\right\|_{p,\alpha}^p & =\int_{\mathbb{B}^n}\left|\int_{\mathbb{B}^n}\frac{a(w)}{\left(1-\langle z,w\rangle\right)^{n+1+\alpha}}dv_{\alpha}(w)\right|^pdv_{\alpha}(z)\\
& \leq\left(\int_{\mathbb{B}^n}\left|\int_{\mathbb{B}^n}\frac{a(w)}{\left(1-\langle z,w\rangle\right)^{n+1+\alpha}}dv_{\alpha}(w)\right|^2dv_{\alpha}(z)\right)^{\frac{p}{2}}\\
&  \lesssim \|a\|_{2,\alpha}^p \leq 1,
\end{split}\ee
since $P_{\alpha}$ is a bounded operator on $L^2(\mathbb{B}_n,dv_{\alpha}).$

If $a$ is supported on $B^{\varrho}(z_{0},r_{0})$ with $z_0\in\mathbb{B}_n$ and $r_0 >0,$ then we have
\be\begin{split}
\int_{B^{\varrho}(z_{0},4 r_{0})}& \left|P_{\alpha}(a)(z)\right|^pdv_{\alpha}(z)\\
\leq & \left(\int_{B^{\varrho}(z_{0},4 r_{0})}\left|P_{\alpha}(a)(z) \right |^2 d v_{\alpha}(z)\right)^{\frac{p}{2}}
v_{\alpha}(B^{\varrho}(z_{0},4 r_{0}))^{1-\frac{p}{2}}\\
\lesssim & \left(\int_{\mathbb{B}_n}|a(z)|^2dv_{\alpha}(z)\right)^{\frac{p}{2}}v_{\alpha}(B^{\varrho}(z_{0},4 r_{0}))^{1-\frac{p}{2}}\\
\leq & v_{\alpha}(B^{\varrho}(z_{0},r_{0}))^{-1}v_{\alpha}(B^{\varrho}(z_{0},r_{0}))^{\frac{p}{2}}v_{\alpha}(B^{\varrho}(z_{0},4 r_{0}))^{1-\frac{p}{2}} \leq C.
\end{split}\ee
Next, we let $Q_k=\{z\in\mathbb{B}_n:\ 2^{k+1}  r_0 \leq \varrho(z,z_0) < 2^{k+2}  r_0\}.$ Then
\be\begin{split}
\int_{\mathbb{B}_n \setminus B^{\varrho}(z_{0},4r_{0})}& \left|P_{\alpha}(a)(z)\right|^pdv_{\alpha}(z)\\
= & \sum_{k=1}^{\infty}\int_{Q_k}\left|P_{\alpha}(a)(z)\right|^pdv_{\alpha}(z)\\
= & \sum_{k=1}^{\infty}\int_{Q_k}\left|\int_{\mathbb{B}_n}\frac{a(w)}{\left(1-\langle z,w\rangle\right)^{n+1+\alpha}}
dv_{\alpha}(w)\right|^pdv_{\alpha}(z)\\
\leq & \sum_{k=1}^{\infty}\int_{Q_k}\left(\|K^{\alpha}(z,\cdot)\|_{\mathcal{S}_N (B^{\varrho}(z_{0},r_{0}))}
v_{\alpha}(B^{\varrho}(z_{0},r_{0}))^{1-\frac{1}{p}}\right)^pdv_{\alpha}(z)\\
= & v_{\alpha}(B^{\varrho}(z_{0},r_{0}))^{p-1}\sum_{k=1}^{\infty}\int_{Q_k}\|K^{\alpha}(z,\cdot)\|_{\mathcal{S}_N (B^{\varrho}(z_{0},r_{0}))}^pdv_{\alpha}(z)
\end{split}\ee
in the above inequality we have used the condition (4) of Definition \ref{df:p-atom}. Notice that, for $z\in Q_k$ and $w\in B^{\varrho}(z_0,r_0),$
\be
\varrho(z,w)\geq\frac{1}{2}\big(\varrho(z,z_0)-2\varrho(w,z_0)\big)>2^{k}r_0-r_0>2^{k-1}r_0
\ee
and so
\be
B^{\varrho}(z_0,2^{k-1}r_0)\subset B^{\varrho}(w,4\varrho(z,w)).
\ee
Hence, by Lemma \ref{le:Kernel} one has, for $z\in Q_k,$
\be\begin{split}
\|K^{\alpha} & (z,\cdot)\|_{\mathcal{S}_N (B^{\varrho}(z_{0},r_{0}))}\\
= & \sum_{|J|= N}r_0^{d(J)}\left \|D^J_{z_{0}} \big(K^{\alpha}(z,\cdot)\big) \right\|_{L^{\infty}(B^{\varrho}(z_{0},r_{0}))}\\
\lesssim & \sum_{|J|= N}r_0^{d(J)}\sup_{w\in B^{\varrho}(z_0,r_0)}\frac{1}{\varrho(z,w)^{d(J)}v_{\alpha}(B^{\varrho}(w,\varrho(z,w)))}\\
\lesssim & \sum_{|J|= N}r_0^{d(J)}\sup_{w\in B^{\varrho}(z_0,r_0)}\frac{1}{(2^{k-1}r_0)^{d(J)}v_{\alpha}(B^{\varrho}(w,4\varrho(z,w)))}\\
\lesssim & \sum_{|J|= N}\frac{1}{(2^{k-1})^{d(J)}v_{\alpha}(B^{\varrho}(z_{0},2^{k-1}r_{0}))}\\
\end{split}\ee
Thus,
\be\begin{split}
\int_{\mathbb{B}_n\setminus B^{\varrho}(z_{0},4r_{0})}&  \left | P_{\alpha}(a)(z) \right |^p dv_{\alpha}(z)\\
\lesssim & v_{\alpha}(B^{\varrho}(z_{0},r_{0}))^{p-1} \sum_{k=1}^{\infty}\int_{Q_k}
\left(\sum_{|J|= N}\frac{1}{(2^{k-1})^{d(J)}v_{\alpha}(B^{\varrho}(z_{0},2^{k-1}r_{0}))}\right)^pdv_{\alpha}(z)\\
\lesssim & v_{\alpha}(B^{\varrho}(z_{0},r_{0}))^{p-1} \sum_{k=1}^{\infty} \sum_{|J|= N}\frac{v_{\alpha}(B^{\varrho}(z_{0},2^{k+1}r_{0}))}{(2^{k-1})^{pd(J)}v_{\alpha}(B^{\varrho}(z_{0},2^{k-1}r_{0}))^{p}}\\
\lesssim & \sum_{k=1}^{\infty}
\frac{v_{\alpha}(B^{\varrho}(z_{0},r_{0}))^{p-1}}{(2^{k-1})^{\frac{pN}{2}}v_{\alpha}(B^{\varrho}(z_{0},2^{k-1}r_{0}))^{p-1}}\\
\lesssim & \sum_{k: 2^{k-1} r_0 \le 1-|z_0|} 2^{- \frac{1}{2}(k-1)p [ N - 2 (n+1)(\frac{1}{p} -1)]}\\
& \quad + \sum^{\8}_{k= k_0:\; 2 > \frac{1-|z_0|}{2^{k_0-2} r_0} \ge 1} 2^{- \frac{1}{2} (k-k_0-1)p [ N - 2 (n+1+ \alpha) (\frac{1}{p} -1) ]}\\
\lesssim & C_{n,p, \alpha, N}
\end{split}\ee
since $N \ge N_{p, \alpha}.$ This completes the proof of Theorem \ref{th:p-atomBergman}.
\hfill$\Box$

\section{Atomic decomposition}\label{atomicdecomp}

This section is devoted to the proof of Theorem \ref{th:atomic-deco}. In section \ref{good-function} we will construct and estimate a collection of smooth bump functions. Sections \ref{AtomConstruction} and \ref{AtomicdecompPf} are devoted to present the construction of atomic decomposition.

\subsection{A partition of unity and good auxiliary functions}\label{good-function}

To prove Theorem \ref{th:atomic-deco}, we need to introduce a smooth partition of unit on any open set $\mathcal{O} \subsetneq \mathbb{B}_n.$ To this end, we present the \textquoteleft Whitney type covering lemma' in our case.

\begin{lemma}\label{le:Whitney}
Let $\mathcal{O} \subsetneq \mathbb{B}_n$ be an open set. Then there are a sequence of balls $\{B^\varrho (z_i,r_i)\}$ in $\mathbb{B}_n,$ positive constants $\mu>1>\nu> \lambda>0$ and $N_0$ depending only on $n$ such that
\begin{enumerate}[{\rm (1)}]

\item for any $i,$
\be
r_i = \frac{1}{2K} \varrho (z_i, \mathcal{O}^c)
\ee
where $K$ is the constant occurring in the quasi-triangular inequality \eqref{eq:quasi-triangularinequa} satisfied by the quasi-metric $\varrho;$

\item $\mathcal{O}=\bigcup_{i}B^\varrho (z_i, \nu r_i);$

\item for each $i,$ $B^\varrho (z_i, \mu r_i)\cap \mathcal{O}^{c}\neq\emptyset;$

\item the balls $B^\varrho (z_i, \lambda r_i)$ are pairwise disjoint;

\item no point in $\mathcal{O}$ lies in more than $N_{0}$ of the ball $B^\varrho (z_i,r_i).$

\end{enumerate}
\end{lemma}

\begin{proof}
See \cite[Lemma 2.4]{Tch2008} for the details.
\end{proof}

Given $z \in \mathbb{B}_n,$ recall that the space of smooth bump functions at $z$ for $\delta >0$ and $L \ge 0$ is denoted by $\mathcal{G}^{L}_{\delta}(z)$ consisting of all functions $g\in C^{\infty}(\mathbb{B}_n)$ for which there exist $z_{0}\in\mathbb{B}_n$ and $r_0 >0$ such that
\be
\mathrm{supp} g \subset B^{\varrho}(z_{0},r_{0}),\quad \varrho(z,z_{0})<\delta r_{0}, \quad \text{and}\; \|g\|_{L,z_{0},r_{0}}\leq 1,
\ee
where
\be
\|g\|_{L,z_{0},r_{0}}=v_{\alpha}(B^{\varrho}(z_0,r_0)) \sup_{|J|\le L}r_{0}^{d(J)} \|D^{J}_{z_0}g\|_{L^{\infty}(B^{\varrho}(z_0,r_0))}.
\ee

\begin{lemma}\label{le:Partition}
Let $\mathcal{O} \subsetneq \mathbb{B}_n $ be an open subset. Then there exist a collection of balls $B^\varrho (z_i,r_i),$ a sequence of functions $\varphi_{i}\in C^{\infty}(\mathbb{B}_n)$ {\rm ($i=1,2,\cdots$)}, and a constant $\mu>1$ depending only on $n,$ such that
\begin{enumerate}[{\rm (1)}]

\item $0\leq\varphi_{i}\leq1;$

\item $\mathrm{supp} \varphi_{i}\subset B^\varrho (z_{i},r_{i});$

\item $\sum_{i=1}^{\infty}\varphi_{i}=\chi_{\mathcal{O}};$

\item for any nonnegative integer $L$ there is a constant $c_{L}>0$ depending only on $L$ and $n$ such that for each $i$ and any $w_{i}\in B^\varrho (z_i, \mu r_i)\cap \mathcal{O}^{c},$
\be
\frac{c_{L}}{\|\varphi_{i}\|_{1,\alpha}} \varphi_{i} \in \mathcal{G}^{L}_{\mu}(w_{i}).
\ee
\end{enumerate}
\end{lemma}

\begin{proof}
Let $B^\varrho_i = B^\varrho (z_i, r_i)$ be given by Lemma \ref{le:Whitney}. For each $i$ we let $\psi_{i}$ be a smooth function satisfying the following conditions:
\begin{enumerate}[$\bullet$]

\item $\psi_{i}$ is supported on the ball $B^{\varrho} ( z_i, r_i );$

\item $0\leq\psi_{i} \leq1,$ and $\psi_{i}\equiv 1$ on $B^d (z_{i},\nu r_{i});$

\item $\|D^{J}_{z_{i}}\psi_{i}\|_{L^{\infty}(B^{\varrho}(z_i, r_i))} \leq c_{J}r_{i}^{-d(J)}$ ($|J|\leq L$).

\end{enumerate}
(See e.g. \cite[p.83]{KL1995} for the construction of $\psi_i.$) Now set
\be
\varphi_{i}(z)=\frac{\psi_{i}(z)}{\sum_{i=1}^{\infty}\psi_{i}(z)},\ \ z\in\mathbb{B}_{n}.
\ee
Note that $1\leq\sum_{i=1}^{\infty}\psi_{i}(z)\leq N_{0},$ the conditions (1), (2), and (3) are clearly satisfied. It remains to check the condition (4). Indeed, by the condition (3) in Lemma \ref{le:Whitney}, for each $i$ we have $B^\varrho (z_i, \mu r_i) \cap \mathcal{O}^{c} \neq \emptyset.$ For any $w_{i} \in B^\varrho (z_i, \mu r_i)\cap\mathcal{O}^{c},$ it is obvious that
\begin{enumerate}[{\rm $\bullet$}]

\item $\varrho (z_i,w_i) < \mu r_i,$

\item $\| D^{J}_{z_i}\varphi_{i} \|_{L^{\infty} (B^{\varrho}(z_{i}, r_{i}))} \leq c_{J} r_{i}^{-d(J)},$

\item $\|\varphi_{i}\|_{1,\alpha} \approx \|\psi_{i}\|_{1,\alpha}\approx v_{\alpha}(B^{\varrho}(z_{i}, r_i)).$
\end{enumerate}
Hence, we have
\be
v_{\alpha}(B^{\varrho}(z_{i},  r_{i})) \sup_{|J|\leq2L} r_{i}^{d(J)}\|D^{J}_{z_i}\varphi_{i}\|_{L^{\infty}(B^{\varrho}(z_{i}, r_{i}))}\leq c_{L}^{-1}\|\varphi_{i}\|_{1,\alpha}.
\ee
Therefore, $\frac{c_{L}}{\|\varphi_{i}\|_{1,\alpha}} \varphi_{i} \in \mathcal{G}^{L}_{\mu}(w_{i}).$
\end{proof}

Let $f\in\mathcal{A}_{\alpha}^{p}$ and $\mu>1$ be the constant appearing in Lemma \ref{le:Whitney}. Given an integer $N \ge N_{p, \alpha},$ let $L \ge \max \big \{ N, \big [ \frac{1}{p} (n+1+ \alpha)\big ] + 1 \big \}$ be an integer. By Lemmas \ref{le:non-tangent} and \ref{le:grand}, we have
\be
\mathcal{K}_{\mu,L}(f)+f^{\star}_{\delta}\in L^p (\mathbb{B}_n, dv_{\alpha} ).
\ee
Let $k_{0}$ be the least integer such that
\be
\left \|\mathcal{K}_{\mu,L}(f)+f^{\star}_{\delta}\right\|_{L^p_{\alpha} (\mathbb{B}_n)} \leq 2^{k_{0}}.
\ee
For any nonnegative integer $k,$ we define
\be
\mathcal{O}_{k}=\left\{z\in\mathbb{B}^n:\; \mathcal{K}_{\mu,L}(f)(z)+f^{\star}_{\delta}(z)>2^{k_{0}+k}\right\}.
\ee
Then $\mathcal{O}_k \subsetneq \mathbb{B}_n$ for any $k =0, 1, \ldots .$ For each $k$ we fix the Whitney type covering $\{B^\varrho (z^k_{i}, r^k_{i})\}_{i=1}^{\infty}$ and the partition of unity $\{\varphi^{k}_{i}\}$ with respect to $\mathcal{O}_{k},$ as constructed in Lemma \ref{le:Partition}.

For each $i$ and $k,$ we denote by $L^{2}_{\alpha, \varphi_{i}^k}(\mathbb{B}_n)$ the $L^2$-space with respect to the probability measure
\be
d v_{\alpha,\varphi_{i}^{k}}: = \frac{\varphi_{i}^{k}}{\|\varphi_{i}^{k}\|_{1,\alpha}}dv_{\alpha}.
\ee
The norm on this space will be denoted by $\|\cdot\|_{\alpha,\varphi_{i}^{k}}$. Then we define a subspace $V^L_{\varphi_{i}^{k}}(z_{i}^{k})$ of $L^{2}_{\alpha, \varphi_{i}^k}(\mathbb{B}_n)$ consisting of \textquoteleft polynomials' of the form
\be
P(z)=\sum_{|J|\leq L}c_{J}\Theta(z_{i}^{k},z)^{J},
\ee
where
\be
\Theta(z_{i}^{k},z)^{J} = \alpha_1^{j_1} \beta_1^{j_2}\cdots \alpha_i^{j_{2i-1}} \beta_i^{j_{2 i}} \cdots \alpha_n^{j_{2n-1}} \beta_n^{j_{2 n}}
\ee
when $\Theta(z_{i}^{k},z) = (\alpha_1, \beta_1, \ldots , \alpha_n, \beta_n)$ and $J = (j_1, j_2, \ldots , j_{2n-1}, j_{2n}).$ It is clear that $V^L_{\varphi_{i}^{k}}(z_{i}^{k})$ is a finite-dimensional Hilbert space.

Let $\pi_{J}(z)$ ($|J|\leq L$) be an orthonormal basis for $V^L_{\varphi_{i}^{k}}(z_{i}^{k}).$

\begin{lemma}\label{le:good-function3}
Let $L$ be a nonnegative integer. Then there is a constant $c_L>0$ depending only on $L$ and $n$ such that
\beq\label{eq:good-function3}
\frac{c_{L}}{\|\varphi_{i}^{k}\|_{1,\alpha}} \pi_{J} \varphi_{i}^{k} \in \mathcal{G}^{L}_{\mu}(w_{i}^{k}),
\eeq
for all $ w_{i}^{k}\in B^\varrho (z_{i}^{k}, \mu r_{i}^{k})\cap\mathcal{O}_{k}^{c}.$
\end{lemma}

\begin{proof}
In order to prove \eqref{eq:good-function3}, for simplicity, by replacing $\Theta(z_{i}^{k},z)^{J}$ with $\Theta(z)^{J},$ we let
\be
\pi_{J}(z)=\sum_{|I|\leq|J|\leq L}a_{J,I}\Theta(z)^{I}.
\ee
Then we will prove that
\beq\label{eq:good-function-inequality1}
|a_{J,I}|\leq c_{J}(r_{i}^{k})^{-d(I)},\ \ |I|\leq |J|\leq L,
\eeq
by mathematical induction, where the positive constant $c_J$ depending only on $J$ and $n.$

To this end, we introduce a linear order $\prec$ on the multi-indices set $\{J: |J| \leq L\}$ such that $|I|<|J|$ implies $I \prec J.$ Note that the orthonormal basis $\pi_{J}$ can be constructed by the Gram-Schmidt process beginning with $\pi_{0} = 1,$ and then
\beq\label{eq:good-function-equality1}
\pi_{J}(z)=\frac{\Theta(z)^{J}-\sum_{I\prec J}\pi_{I}(z)\int\Theta(w)^{J}\pi_{I}(w)dv_{\alpha,\varphi_{i}^{k}}(w)}{\left\|\Theta(z)^{J}-\sum_{I\prec J}\pi_{I}(z)\int\Theta(w)^{J}\pi_{I}(w)dv_{\alpha,\varphi_{i}^{k}}(w)\right\|_{\alpha,\varphi_{i}^{k}}}.
\eeq
Using mathematical induction, we assume that if $O \preceq I \prec J$ then
\be
|a_{I,O}|\leq c_{I}(r_{i}^{k})^{-d(O)}.
\ee
Because, for any $O\preceq I$ we have $\left|\Theta(z)^{O}\varphi_{i}^{k}(z)\right|\leq(r_{i}^{k})^{d(O)}$, it follows that $\pi_{I}(z)\varphi_{i}^{k}\leq c_{I}.$ Therefore, in the numerator of (\ref{eq:good-function-equality1}), the coefficient of $\Theta(z)^{I}$ ($I\preceq J$) is dominated by $c_{J}(r_{i}^{k})^{d(J)-d(I)}.$

In the following, we shall estimate the denominator of (\ref{eq:good-function-equality1}). Recalling the constant $\nu$ in Lemma \ref{le:Whitney}, we claim that there exists a constant $C=C(\nu, n)$ such that
\beq\label{eq:good-function-inequality2}
\frac{(1-|z|)^{\alpha}}{v_{\alpha}(B^{\varrho}(z_{i}^{k},r_{i}^{k}))}\gtrsim\frac{C}{(r_{i}^{k})^{n+1}},
\eeq
for all $z \in \{w \in\mathbb{B}_n:\; \varrho (z_{i}^{k},w)<\frac{1}{4}\nu r_{i}^{k}\}.$

First of all, we note that
\be
1-|z_{i}^{k}| = \varrho \big ( \frac{z^k_i}{|z^k_i|}, z^k_i \big )> \nu r_{i}^{k}.
\ee
Then, the proof of (\ref{eq:good-function-inequality2}) can be divided into two cases: $0\leq\alpha<\infty$ and $-1<\alpha<0.$

{\it $\bullet$\; Case $0\leq\alpha<\infty$}.\; Suppose $z\in\{w \in\mathbb{B}_n: \; \varrho (z_{i}^{k},w)<\frac{1}{4} \nu r_{i}^{k}\}.$ If $r_{i}^{k}\geq1-|z_{i}^{k}|,$ we have
\be\begin{split}
1-|z|&=1-|z_{i}^{k}|+|z_{i}^{k}|-|z| \ge  \nu r_{i}^{k}-\varrho(z_{i}^{k},z)
\ge \frac{1}{2} \nu r_{i}^{k},
\end{split}\ee
therefore,
\be
\frac{(1-|z|)^{\alpha}}{v_{\alpha}(B^{\varrho}(z_{i}^{k},r_{i}^{k}))}\gtrsim\frac{C (r_{i}^{k})^{\alpha}}{(r_{i}^{k})^{n+1+\alpha}}\geq\frac{C}{(r_{i}^{k})^{n+1}};
\ee
if $r_{i}^{k}<1-|z_{i}^{k}|,$ we also have
\be\begin{split}
1-|z|&=1-|z_{i}^{k}|+|z_{i}^{k}|-|z| \ge 1-|z_{i}^{k}|-\varrho(z_{i}^{k},z) \ge 1-|z_{i}^{k}|-\frac{1}{4} \nu r_{i}^{k}\\
& \ge 1-|z_{i}^{k}|-\frac{1}{4} \nu (1-|z_{i}^{k}|) \ge \Big ( 1-\frac{1}{4}\nu \Big )(1-|z_{i}^{k}|),
\end{split}\ee
hence,
\be
\frac{(1-|z|)^{\alpha}}{v_{\alpha}(B^{\varrho}(z_{i}^{k},r_{i}^{k}))}\gtrsim\frac{C (1-|z_{i}^{k}|)^{\alpha}}{(r_{i}^{k})^{n+1}(1-|z_{i}^{k}|)^{\alpha}}\geq\frac{C}{(r_{i}^{k})^{n+1}}.
\ee

{\it $\bullet$\; Case $-1 < \alpha <0$}.\; Let $z\in\{w\in\mathbb{B}_n: \; \varrho (z_{i}^{k},w)<\frac{1}{4} \nu r_{i}^{k}\}.$ If $r_{i}^{k}\geq1-|z_{i}^{k}|,$ we have
\be\begin{split}
1-|z| \le 1-|z_{i}^{k}|+ \varrho (z_{i}^{k}, z) \lesssim r_{i}^{k},
\end{split}\ee
therefore,
\be
\frac{(1-|z|)^{\alpha}}{v_{\alpha}(B^{\varrho}(z_{i}^{k},r_{i}^{k}))}\gtrsim \frac{ (r_{i}^{k})^{\alpha}}{(r_{i}^{k})^{n+1+\alpha}}\geq\frac{C}{(r_{i}^{k})^{n+1}};
\ee
if $r_{i}^{k}<1-|z_{i}^{k}|,$ we also have
\be\begin{split}
1-|z| \le 1-|z_{i}^{k}|+ \varrho (z_{i}^{k}, z) \lesssim 1-|z_{i}^{k}|,
\end{split}\ee
hence,
\be
\frac{(1-|z|)^{\alpha}}{v_{\alpha}(B^{\varrho}(z_{i}^{k},r_{i}^{k}))}\gtrsim\frac{C (1-|z_{i}^{k}|)^{\alpha}}{(r_{i}^{k})^{n+1}(1-|z_{i}^{k}|)^{\alpha}}\geq\frac{C}{(r_{i}^{k})^{n+1}}.
\ee
In summery, \eqref{eq:good-function-inequality2} is obtained.

We now come back to estimate the denominator of (\ref{eq:good-function-equality1}). Let
\be
F_{tr_{i}^{k}}=\left\{(\alpha,\beta)\triangleq(\alpha_1,\cdots,\beta_{n}):
|\alpha_1|,|\beta_1|<tr_{i}^{k};\alpha_{2}^2+\cdots+\beta_{n}^{2}<tr_{i}^{k}\right\}.
\ee
Then, we have
\be\begin{split}
\Big \| & \Theta(z)^{J} - \sum_{I\prec J}\pi_{I}(z)\int\Theta(w)^{J}\pi_{I}(w)dv_{\alpha,\varphi_{i}^{k}}(w) \Big \|_{\alpha,\varphi_{i}^{k}}^{2}\\
= & \frac{1}{ \| \varphi_{i}^{k} \|_{1, \alpha}} \int \Big | \Theta(z)^{J}-\sum_{I\prec J}\pi_{I}(z)\int\Theta^{J}\pi_{I}dv_{\alpha,\varphi_{i}^{k}} \Big |^{2} \varphi_{i}^{k} d v_{\alpha} (z)\\
\gtrsim & \frac{1}{v_{\alpha}(B^{\varrho}(z_{i}^{k},r_{i}^{k}))} \int_{B^\varrho (z_{i}^{k},\nu r_{i}^{k})} \Big | \Theta(z)^{J}-\sum_{I\prec J}\pi_{I}(z)\int\Theta^{J}\pi_{I}dv_{\alpha,\varphi_{i}^{k}} \Big |^{2}\varphi_{i}^{k} d v_{\alpha} (z)\\
\ge & \frac{1}{v_{\alpha}(B^{\varrho}(z_{i}^{k},r_{i}^{k}))} \int_{\{z\in\mathbb{B}_n:\; \varrho (z_{i}^{k},z) < \frac{1}{4} \nu r_{i}^{k}\}} \Big | \Theta(z)^{J} - \sum_{I\prec J}\pi_{I}(z)\int\Theta^{J}\pi_{I}dv_{\alpha,\varphi_{i}^{k}} \Big |^{2} d v_{\alpha}(z)\\
\gtrsim & \frac{1}{(r^{k}_{i})^{n+1}} \int_{\left\{z\in\mathbb{B}_n: \; \varrho (z_{i}^{k},z)<\frac{1}{4} \nu r_{i}^{k}\right\}} \Big | \Theta(z)^{J}-\sum_{I\prec J}\pi_{I}(z)\int\Theta^{J}\pi_{I}dv_{\alpha,\varphi_{i}^{k}} \Big |^{2}dv(z)\\
= & \frac{1}{(r^{k}_{i})^{n+1}} \int_{F_{\frac{1}{4} \nu r_{i}^{k}}} \Big | (\alpha,\beta)^{J}-\sum_{I\prec J}\pi_{I}(\alpha,\beta)\int\Theta^{J}\pi_{I}dv_{\alpha,\varphi_{i}^{k}} \Big |^{2}dv(\alpha,\beta)\\
= & \int_{F_{\frac{1}{4} \nu}}\Bigg | \sum_{I \prec J} \pi_{I} \big ( r_{i}^{k} \alpha_1,r_{i}^{k} \beta_1,(r_{i}^{k})^{\frac{1}{2}} \alpha_2,\cdots,(r_{i}^{k})^{\frac{1}{2}} \beta_n \big )\\
& \quad \quad \quad \times \int\Theta^{J}\pi_{I}dv_{\alpha,\varphi_{i}^{k}}- (r_{i}^{k})^{d(J)}(\alpha,\beta)^{J} \Bigg|^{2} dv(\alpha,\beta),
\end{split}\ee
where we have used (\ref{eq:good-function-inequality2}) and made the change of variables
\be
\left(\alpha_1,\beta_1,\alpha_2,\beta_2,\cdots,\alpha_n,\beta_n\right)\rightarrow\left(\frac{\alpha_1}{r_{i}^{k}},\frac{\beta_1}{r_{i}^{k}},
\frac{\alpha_2}{(r_{i}^{k})^{\frac{1}{2}}},\frac{\beta_2}{(r_{i}^{k})^{\frac{1}{2}}},\cdots,
\frac{\alpha_n}{(r_{i}^{k})^{\frac{1}{2}}},\frac{\beta_n}{(r_{i}^{k})^{\frac{1}{2}}}\right).
\ee
We continue to estimate the last integral, which is equal to
\be\begin{split}
(r_{i}^{k})^{2d(J)} \int_{F_{\frac{1}{4} \nu}}& \Bigg | \sum_{I \prec J} \frac{1}{(r_{i}^{k})^{d(J)}} \pi_{I} \big ( r_{i}^{k}\alpha_1,r_{i}^{k}\beta_1,(r_{i}^{k})^{\frac{1}{2}}\alpha_2,\cdots,(r_{i}^{k})^{\frac{1}{2}}\beta_n  \big )\\
&\quad \quad \quad \times \int \Theta^{J}\pi_{I}dv_{\alpha,\varphi_{i}^{k}} - (\alpha,\beta)^{J}\Bigg|^{2} d v(\alpha,\beta)\\
& \gtrsim (r_{i}^{k})^{2d(J)}\int_{F_{\frac{1}{m}c\nu}}|(\alpha,\beta)^{J}-P_{J}(\alpha,\beta)|^2dv(\alpha,\beta)\\
& \geq c_J (r_{i}^{k})^{2d(J)},
\end{split}\ee
where $P_{J}(\alpha,\beta)$ is the projection of $(\alpha,\beta)^{J}$ into the Hilbert space of polynomials $P(\alpha,\beta)$ spanned by $\{(\alpha,\beta)^{I}:I\prec J\}$ with the norm
\be
\|P\|= \Bigg ( \int_{F_{\frac{1}{4} \nu}}|P(\alpha,\beta)|^2dv(\alpha,\beta) \Bigg)^{\frac{1}{2}}.
\ee
Combining this estimation on the denominator of (\ref{eq:good-function-equality1}) with the previous estimation for its numerator yields that the coefficient of $\Theta(z)^{I}$ ($I\preceq J$) is dominated by $c_{J}(r_{i}^{k})^{-d(I)},$ i.e.,
\be
|a_{J,I}|\leq c_{J}(r_{i}^{k})^{d(J)-d(I)}(r_{i}^{k})^{-d(J)}=c_{J}(r_{i}^{k})^{-d(I)},\ \ I\preceq J.
\ee
Therefore, the claim (\ref{eq:good-function-inequality1}) is proved.

Now we return to the proof of \eqref{eq:good-function3}. Indeed, by the Liebnitz rule and the fact $\frac{c_{L}}{\|\varphi_{i}\|_{1,\alpha}}\varphi_{i}\in\mathcal{G}^{L}_{\mu}(w_{i})$, we have
\be
\frac{c_{L}(r_{i}^{k})^{-d(I)}}{\|\varphi_{i}\|_{1,\alpha}}\Theta(z_{i}^{k},z)^{I}\varphi_{i}\in\mathcal{G}^{L}_{\mu}(w_{i}),\ \ |I|\leq L.
\ee
Thus, by (\ref{eq:good-function-inequality1}) we also have
\be
\frac{c_{L}}{\|\varphi_{i}^{k}\|_{1,\alpha}} \pi_{J} \varphi_{i}^{k}\in\mathcal{G}^{L}_{\mu}(w_{i}),\ \ |J|\leq L.
\ee
This completes the proof.
\end{proof}

 We denote by $\mathcal{P}_{\varphi_{i}^k}$ the orthogonal projection of $L^{2}_{\varphi_{i}^k}(\mathbb{B}_n)$ onto $V_{\varphi_{i}^{k}}(z_{i}^{k}).$

\begin{lemma}\label{le:good-function}
With the notation introduced above, there exists a constant $C>0$ such that for $f \in\mathcal{A}^p_{\alpha},$
\beq\label{eq:good-function1}
\left|\mathcal{P}_{\varphi_{i}^k}(f)(z)\varphi_{i}^k(z)\right|\leq C 2^{k_{0}+k},
\eeq
and
\beq\label{eq:good-function2}
\left | \mathcal{P}_{\varphi_{j}^{k+1}} \Big ( [ f-\mathcal{P}_{\varphi_{j}^{k+1}}(f) ] \varphi_{i}^{k}\Big)(z)\varphi_{j}^{k+1}(z)\right|\leq C 2^{k_{0}+k+1},
\eeq
for all $i,j,k.$
\end{lemma}

\begin{proof}
We first prove \eqref{eq:good-function1}. In fact, by (\ref{eq:good-function3}) we have
\be\begin{split}
\Big | \mathcal{P}_{\varphi_{i}^k} & (f)(z) \varphi_{i}^k(z) \Big |\\
=& \left | \left( \sum_{|J| \leq L} \int_{\mathbb{B}_n} f(w) \overline{\pi_{J}(w)} \frac{\varphi_{i}^{k}(w)}{\|\varphi_{i}^{k}\|_{1,\alpha}}dv_{\alpha}(w) \pi_{J}(z) \right) \phi_{i}^{k}(z) \right |\\
\lesssim & \mathcal{K}_{\mu,L}(f)(w_{i}^{k})\sum_{|J|\leq L}\left|\pi_{J}(z)\varphi_{i}^k(z)\right|\\
=& \mathcal{K}_{\mu,L}(f)(w_{i}^{k}) \sum_{|J|\leq L} \frac{c_{L}^{-1} \|\varphi_{i}^k \|_{1,\alpha}}{v_{\alpha}(B^{\varrho}(z_{i}^{k}, r_{i}^{k}))}\\
& \quad \times \left | \frac{c_{L} v_{\alpha}(B^{\varrho}(z_{i}^{k}, r_{i}^{k}))}{\|\varphi_{i}^k\|_{1,\alpha}} \pi_{J}(z) \varphi_{i}^k(z) \right |\\
\lesssim & 2^{k_{0}+k}\sum_{|J|\leq L} \left | \frac{c_{L} v_{\alpha}(B^{\varrho}(z_{i}^{k}, r_{i}^{k}))}{\|\varphi_{i}^k\|_{1,\alpha}} \pi_{J}(z) \varphi_{i}^k(z)
\right |\\
\lesssim & c_{L} 2^{k_{0}+k},
\end{split}\ee
the last inequality is also the consequence of (\ref{eq:good-function3}).

In what follows, we prove the second estimate (\ref{eq:good-function2}) which follows from (\ref{eq:good-function1}) by the same argument. For simplicity, we still use $\pi_{J}(z)$ ($|J|\leq L$) to denote an orthonormal basis for $V_{\varphi_{j}^{k+1}}(z_{j}^{k+1}).$ Then
\be\begin{split}
\bigg |&  \mathcal{P}_{\varphi_{j}^{k+1}}\Big ( [f-\mathcal{P}_{\varphi_{j}^{k+1}}(f) ] \varphi_{i}^{k}\Big)(z)\varphi_{j}^{k+1}(z)\bigg | \\
= & \Bigg | \sum_{|J| \leq L} \int_{\mathbb{B}_n} \Big ( [f-\mathcal{P}_{\varphi_{j}^{k+1}}(f) ] \varphi_{i}^{k} \Big )(w) \overline{\pi_{J}(w)} \frac{\varphi_{j}^{k+1}(w)}{\|\varphi_{j}^{k+1}\|_{1,\alpha}}dv_{\alpha}(w) \pi_{J}(z) \phi_{j}^{k+1}(z) \Bigg |\\
\leqslant & \left | \sum_{|J| \leq L} \int_{\mathbb{B}_n}  f(w)\varphi_{i}^{k}(w) \overline{\pi_{J}(w)} \frac{\varphi_{j}^{k+1}(w)}{\|\varphi_{j}^{k+1}\|_{1,\alpha}}dv_{\alpha}(w) \pi_{J}(z) \phi_{j}^{k+1}(z) \right |\\
& \quad + \left | \sum_{|J| \leq L} \int_{\mathbb{B}_n}  \mathcal{P}_{\varphi_{j}^{k+1}}(f) (w)\varphi_{i}^{k}(w) \overline{\pi_{J}(w)} \frac{\varphi_{j}^{k+1}(w)}{\|\varphi_{j}^{k+1}\|_{1,\alpha}}dv_{\alpha}(w) \pi_{J}(z) \phi_{j}^{k+1}(z) \right |\\
=: & I + II.
\end{split}\ee
Notice that the $\pi_{J}(z)$ are uniformly bounded on the support of $\phi_{j}^{k+1}(z)$ and
\be
\left|\mathcal{P}_{\varphi_{j}^{k+1}}(f)(z)\varphi_{j}^{k+1}(z)\right|\leq C 2^{k_{0}+k+1},
\ee
we have $II\lesssim2^{k_{0}+k+1}.$

To estimate $I,$ we note that if each integration part is not zero in $I,$ then $\varphi_{i}^{k}\varphi_{j}^{k+1}(z)\not\equiv0$ and hence there exist a constant $C$ such that $r_{j}^{k+1}\leq Cr_{i}^{k}.$ In fact, this can be verified as follows.

Let $K$ be the constant occurring in the quasi-triangle inequality satisfied by $\varrho.$ By Lemma \ref{le:Whitney}, we know that
\be
r_{i}^{k}=\frac{1}{2K} \varrho (z_{i}^{k},\mathcal{O}^{c}_{k}),
\ee
for all $i$ and $k.$ Let $w\in B^\varrho (z_i^{k},r_i^{k})\cap B^\varrho (z_{j}^{k+1},r_{j}^{k+1}).$ Since $\mathcal{O}_{k+1}\subset\mathcal{O}_{k},$ we have
\be\begin{split}
r_{j}^{k+1}& = \frac{1}{2K} \varrho (z_{j}^{k+1},\mathcal{O}^{c}_{k+1}) \le \frac{1}{2K} \cdot K \big[ \varrho (z_{j}^{k+1},w)+ \varrho (w,\mathcal{O}^{c}_{k+1})\big]\\
& \le \frac{1}{2}r_{j}^{k+1}+\frac{1}{2} \varrho (w,\mathcal{O}^{c}_{k}) \le \frac{1}{2}r_{j}^{k+1}+\frac{1}{2} K \big [ \varrho (w,z_{i}^{k})+ \varrho (z_{i}^{k},\mathcal{O}^{c}_{k})\big].
\end{split}\ee
Then,
\be\begin{split}
r_{j}^{k+1}\le & K r_{i}^{k}+K \varrho (z_{i}^{k},\mathcal{O}^{c}_{k})
\le K r_{i}^{k}+2K^{2} \frac{ \varrho (z_{i}^{k},\mathcal{O}^{c}_{k})}{2K}\\
= & K r_{i}^{k}+2K^{2} r_{i}^{k} \le K (1 + 2 K) r_i^k.
\end{split}\ee
Thus, there exist a constant $C$ such that $r_{j}^{k+1}\leq Cr_{i}^{k}.$

Now, by (\ref{eq:good-function3}) and the change of variable and Liebnitz rule, we have
\be
\sup_{|J|\leq L} (r_{j}^{k+1})^{d(J)}\|D^{J}_{z_j^{k+1}}\varphi_{i}^{k}\pi_{J}\varphi_{j}^{k+1}\|_
{L^{\infty}(B^{\varrho}(z_{j}^{k+1}, r_{j}^{k+1}))}\leq c_{L}^{-1},
\ee
where the constant $c_{L}$ depending only on $L.$ This also implies that
\be
\frac{c_{L}}{\|\varphi_{j}^{k+1}\|_{1,\alpha}} \varphi_{i}^{k} \pi_{J} \varphi_{j}^{k+1} \in \mathcal{G}^{L}_{\mu}(w_{j}^{k+1}),
\ee
for each $w_{j}^{k+1}\in B^\varrho (z_{j}^{k+1},\mu r_{j}^{k+1})\cap\mathcal{O}_{k+1}^{c}.$ Thus, we have
\be\begin{split}
I=& \left | \sum_{|J| \leq L} \int_{\mathbb{B}_n}  f(w)\varphi_{i}^{k}(w) \overline{\pi_{J}(w)} \frac{\varphi_{j}^{k+1}(w)}{\|\varphi_{j}^{k+1}\|_{1,\alpha}}dv_{\alpha}(w) \pi_{J}(z) \phi_{j}^{k+1}(z) \right |\\
\lesssim & \mathcal{K}_{\mu,L}(f)(w_{j}^{k+1})\sum_{|J|\leq L}\left|\pi_{J}(z)\varphi_{j}^{k+1}(z)\right|\\
\lesssim & 2^{k_{0}+k+1}.
\end{split}\ee
Hence \eqref{eq:good-function2} is proved and the proof of Lemma \ref{le:good-function} is complete.
\end{proof}

\subsection{Construction of atoms}\label{AtomConstruction}

In this subsection, we will give the construction of atomic decomposition for each $f\in\mathcal{A}_{\alpha}^{p}$ ($0<p\leq1$). Given an integer $N \ge N_{p, \alpha},$ let $L \ge \max \big \{ N, \big [ \frac{1}{p} (n+1+ \alpha)\big ] + 1 \big \}$ be an integer. Recall that
\be
\mathcal{O}_{k}=\left\{z\in\mathbb{B}^n :\; \mathcal{K}_{\mu,L}(f)(z)+f^{\star}_{\delta}(z)>2^{k_{0}+k}\right\},\quad k=0,1, \ldots.
\ee
For each $k$ we fix the Whitney type covering $\{B^\varrho (z^k_{i},r^k_{i})\}_{i=1}^{\infty}$ and the partition of unity $\{\varphi^{k}_{i}\}$ with respect to $\mathcal{O}_{k},$ as constructed in Lemma \ref{le:Partition}. Then, we can write
\be\begin{split}
f = & \Big ( f - \sum_{i=1}^{\infty}f\varphi_{i}^{k} \Big ) + \sum_{i=1}^{\infty}f\varphi_{i}^{k}\\
= & f_{k}+\sum_{i=1}^{\infty}f\varphi_{i}^{k}\\
= & f_{k}+\sum_{i=1}^{\infty}\mathcal{P}_{\varphi_{i}^{k}}(f)\varphi_{i}^{k}+\sum_{i=1}^{\infty}\left(f-\mathcal{P}_{\varphi_{i}^{k}}(f)\right)\varphi_{i}^{k}\\
= & h_{k}+\sum_{i=1}^{\infty}\left(f-\mathcal{P}_{\varphi_{i}^{k}}(f)\right)\varphi_{i}^{k},
\end{split}\ee
where $f_k = f - \sum_{i=1}^{\infty}f\varphi_{i}^{k}$ and
\beq\label{eq:h-fdecomp}
h_{k}= \Big ( f-\sum_{i=1}^{\infty}f\varphi_{i}^{k} \Big ) + \sum_{i=1}^{\infty}\mathcal{P}_{\varphi_{i}^{k}}(f)\varphi_{i}^{k}.
\eeq
Notice that by \eqref{eq:good-function1},
\beq\label{eq:P(f)}
\left|\sum_{i=1}^{\infty}\mathcal{P}_{\varphi_{i}^{k}}(f)(z)\varphi_{i}^{k}(z)\right|\leq c2^{k_0+k},
\eeq
because no point in $\mathcal{O}_k$ lies in more than $N_0$ of the balls $B^\varrho (z_{i}^{k},r_{i}^{k}).$ Moreover,
\be
\mathrm{supp} \left ( \sum_{i=1}^{\infty} \left [ f-\mathcal{P}_{\varphi_{i}^{k}}(f) \right ] \varphi_{i}^{k} \right ) \subset \mathcal{O}_k.
\ee
This implies that $\sum_{i=1}^{\infty} \left [ f-\mathcal{P}_{\varphi_{i}^{k}}(f) \right ] \varphi_{i}^{k} \to 0$ as $k \to \8.$ Hence, by \eqref{eq:h-fdecomp} one concludes that
\be
f-h_{k}\rightarrow 0 \; \text{as}\; k \to \8\; \text{for a.e.}\; z\in \mathbb{B}_n.
\ee
This follows that
\beq\label{eq:f-hdecomp}
f=h_{0}+\sum_{k=0}^{\infty} \big ( h_{k+1}- h_{k} \big ),\quad \mathrm{a.e} \quad z\in \mathbb{B}_n.
\eeq

Now, since
\be
f-h_{k} = \sum_{i=1}^{\infty} \Big ( f-\mathcal{P}_{\varphi_{i}^{k}}(f) \Big ) \varphi_{i}^{k}
\ee
and
\be\begin{split}
\sum_{i=1}^{\infty} & \mathcal{P}_{\varphi_{j}^{k+1}}\Big ( \big [ f-\mathcal{P}_{\varphi_{j}^{k+1}}(f) \big ] \varphi_{i}^{k} \Big ) = \mathcal{P}_{\varphi_{j}^{k+1}} \Big ( \big [ f-\mathcal{P}_{\varphi_{j}^{k+1}}(f) \big ] \chi_{\mathcal{O}_k} \Big )\\
& = \mathcal{P}_{\varphi_{j}^{k+1}} \big [ f \chi_{\mathcal{O}_k} \big ] - \mathcal{P}_{\varphi_{j}^{k+1}} \big [ \mathcal{P}_{\varphi_{j}^{k+1}}(f) \chi_{\mathcal{O}_k}\big ] = 0,
\end{split}\ee
we can write
\be\begin{split}
& h_{k+1}-h_{k} \\
& = \big ( f-h_{k} \big ) - \big ( f-h_{k+1} \big )\\
& = \sum_{i=1}^{\infty} \big [ f-\mathcal{P}_{\varphi_{i}^{k}}(f) \big ] \varphi_{i}^{k}-
\sum_{j=1}^{\infty} \big [ f-\mathcal{P}_{\varphi_{j}^{k+1}}(f) \big ] \varphi_{j}^{k+1}\\
& = \sum_{i=1}^{\infty} \big [ f-\mathcal{P}_{\varphi_{i}^{k}}(f) \big ] \varphi_{i}^{k}\\
& \quad \quad - \sum_{j=1}^{\infty}\sum_{i=1}^{\infty} \bigg \{ \big [ f-\mathcal{P}_{\varphi_{j}^{k+1}}(f) \big ] \varphi_{i}^{k}- \mathcal{P}_{\varphi_{j}^{k+1}} \Big( \big [ f-\mathcal{P}_{\varphi_{j}^{k+1}}(f)\big ] \varphi_{i}^{k} \Big ) \bigg \} \varphi_{j}^{k+1}\\
& = \sum_{i=1}^{\infty} \bigg \{ \big [ f- \mathcal{P}_{\varphi_{i}^{k}}(f) \big ] \varphi_{i}^{k}\\
& \quad \quad - \sum_{j=1}^{\infty} \Big ( \big [ f-\mathcal{P}_{\varphi_{j}^{k+1}}(f) \big ] \varphi_{i}^{k}-\mathcal{P}_{\varphi_{j}^{k+1}} \Big ( \big [ f-\mathcal{P}_{\varphi_{j}^{k+1}}(f) \big ] \varphi_{i}^{k} \Big ) \Big ) \varphi_{j}^{k+1} \bigg \}\\
& =: \sum_{i=1}^{\infty}b_{i}^{k},
\end{split}\ee
where
\beq\label{eq:b-decomp}
\begin{split}
b_{i}^{k}=& \big [ f- \mathcal{P}_{\varphi_{i}^{k}}(f) \big ] \varphi_{i}^{k}\\
& - \sum_{j=1}^{\infty} \Big \{ \big [ f-\mathcal{P}_{\varphi_{j}^{k+1}}(f) \big ] \varphi_{i}^{k}-\mathcal{P}_{\varphi_{j}^{k+1}} \Big ( \big [ f-\mathcal{P}_{\varphi_{j}^{k+1}}(f) \big ] \varphi_{i}^{k} \Big ) \Big \} \varphi_{j}^{k+1}.
\end{split}\eeq
Consequently, we can write formally
\beq\label{eq:f-bdecomp}
f = h_{0} + \sum_{k=0}^{\infty}\sum_{i=1}^{\infty}b_{i}^{k}
\eeq
where the convergence of the series in the right hand side will be fixed below.

Now set
\be
a_{0}=\frac{1}{\lambda_{0}}h_{0},\quad a_{i}^{k} = \frac{1}{\lambda_{i}^{k}} b_{i}^{k},
\ee
where
\be
\lambda_{0} = \|h_{0}\|_{L^{\infty} (\mathbb{B}_n)},\qquad \lambda_{i}^{k} = 2^{k_{0}+k+1}v_{\alpha}(B^{\varrho}(z_{i}^{k},C r_{i}^{k}))^{\frac{1}{p}},
\ee
and $ C = C(n, K)$ is a constant which will be fixed later. Therefore, \eqref{eq:f-bdecomp} can be rewritten as
\beq\label{eq:atomicdecomp}
f=\lambda_{0}a_{0}+\sum_{k=0}^{\infty}\sum_{i=1}^{\infty}\lambda_{i}^{k}a_{i}^{k}.
\eeq
In the sequel, we will check that this representation of $f$ is the desired atomic decomposition, where the series in the right hand side converges in the sense of distributions.

I. {\it Support of the $b_{i}^{k}$'s.}\; We note that the first term in \eqref{eq:b-decomp} is clearly supported in $B^\varrho (z_{i}^{k},r_{i}^{k}).$ Note that if the terms in the series \eqref{eq:b-decomp} are not identically $0,$ then the condition
\be
B^\varrho (z_{i}^{k},r_{i}^{k})\cap B^\varrho (z_{j}^{k+1},r_{j}^{k+1})\neq\emptyset
\ee
must be satisfied for some $j.$ By a standard argument (see the proof of the second inequality \eqref{eq:good-function2} in Lemma \ref{le:good-function}), we know that there is a constant $C_1>0$ depending only on $\varrho$ such that $r_j^{k+1} \leq C_1 r_i^k.$ Hence, there exists a constant $C>0$ depending only on $K$ and $C_1$ such that
\be
B^{\varrho}(z_{j}^{k+1}, r_j^{k+1} ) \subset B^{\varrho}(z_{i}^{k}, C {r_{i}^{k}}).
\ee
Thus $b_{i}^{k}$ is supported in $B^{\varrho}(z_{i}^{k},C{r_{i}^{k}})$ and so does $a_i^k.$

II. {\it Size estimates for $h_{0}$ and $b_{i}^{k}$'s.}\; Firstly, by \eqref{eq:h-fdecomp} and \eqref{eq:P(f)} we have
\be\begin{split}
|h_{0}| = & \bigg | f_{0}+\sum_{i=1}^{\infty}\mathcal{P}_{\varphi_{i}^{0}}(f)\varphi_{i}^{0} \bigg |
= \bigg | f \chi_{\mathcal{O}_{0}^{c}}+\sum_{i=1}^{\infty}\mathcal{P}_{\varphi_{i}^{0}}(f)\varphi_{i}^{0} \bigg |\\
=& \bigg | f\chi_{\mathcal{O}_{0}^{c}}+\sum_{i=1}^{\infty}\mathcal{P}_{\varphi_{i}^{0}}(f)\varphi_{i}^{0} \bigg |
\leq \|f_{\delta}^{\star}\|_{L^{\infty}({\mathcal{O}_{0}^{c}})}+ \bigg | \sum_{i=1}^{\infty}\mathcal{P}_{\varphi_{i}^{0}}(f)\varphi_{i}^{0} \bigg |
\leq c 2^{k_{0}}.
\end{split}\ee
Thus $\| h_{0} \|_{L^{\8}} \le c 2^{k_0},$ so $a_{0}$ is a $(p, \8, N)_{\alpha}$-atom.

On the other hand, by \eqref{eq:b-decomp} we have
\be\begin{split}
|b_{i}^{k}| = & \bigg | \big [ f-\mathcal{P}_{\varphi_{i}^{k}}(f) \big ]\varphi_{i}^{k}\\
& \quad - \sum_{j=1}^{\infty} \bigg ( \big [ f-\mathcal{P}_{\varphi_{j}^{k+1}}(f) \big] \varphi_{i}^{k}-
{P}_{\varphi_{j}^{k+1}} \Big ( \big [ f-\mathcal{P}_{\varphi_{j}^{k+1}}(f) \big ] \varphi_{i}^{k} \Big ) \bigg ) \varphi_{j}^{k+1}\bigg |\\
\leq & \bigg | \big [ f-\mathcal{P}_{\varphi_{i}^{k}}(f) \big ] \varphi_{i}^{k}-
\sum_{j=1}^{\infty} \big [ f-\mathcal{P}_{\varphi_{j}^{k+1}}(f) \big ] \varphi_{i}^{k}\varphi_{j}^{k+1} \bigg |\\
& \qquad + \bigg | \sum_{j=1}^{\infty}{P}_{\varphi_{j}^{k+1}} \Big ( \big [ f-\mathcal{P}_{\varphi_{j}^{k+1}}(f) \big ] \varphi_{i}^{k} \Big ) \varphi_{j}^{k+1}\bigg|.\\
\end{split}\ee
The second term on the right hand side is bounded by $c 2^{k_0 + k + 1}$ by \eqref{eq:P(f)}, while the first term is equal to
\be\begin{split}
\bigg |\big [ f-\mathcal{P}_{\varphi_{i}^{k}}(f) \big ] & \varphi_{i}^{k} \chi_{\mathcal{O}_{k+1}^{c}} + \sum_{j=1}^{\infty} \left ( \big [ f-\mathcal{P}_{\varphi_{i}^{k}}(f) \big ] - \big [ f-\mathcal{P}_{\varphi_{j}^{k+1}}(f) \big ] \right ) \varphi_{i}^{k}\varphi_{j}^{k+1} \bigg |\\
= & \bigg |\big [ f-\mathcal{P}_{\varphi_{i}^{k}}(f) \big ] \varphi_{i}^{k} \chi_{\mathcal{O}_{k}\setminus \mathcal{O}_{k+1}} + \sum_{j=1}^{\infty} \left ( \mathcal{P}_{\varphi_{j}^{k+1}}(f)-\mathcal{P}_{\varphi_{i}^{k}}(f) \right )\varphi_{i}^{k} \varphi_{j}^{k+1} \bigg |\\
\leq & \big | f \chi_{\mathcal{O}_{k}\setminus\mathcal{O}_{k+1}} \big | + \big |\mathcal{P}_{\varphi_{i}^{k}}(f)\varphi_{i}^{k} \big |+ \bigg | \sum_{j=1}^{\infty} \mathcal{P}_{\varphi_{j}^{k+1}}(f) \varphi_{j}^{k+1} \bigg | + \big | \mathcal{P}_{\varphi_{i}^{k}}(f)\varphi_{i}^{k} \big |\\
\leq & f^{\star}_{\delta}\chi_{\mathcal{O}_{k}\setminus\mathcal{O}_{k+1}} + c 2^{k_{0}+k+1}\\
\lesssim & 2^{k_{0}+k+1},
\end{split}\ee
where we have used Lemma \ref{le:good-function}. Thus, $|b_{i}^{k}| \lesssim 2^{k_{0}+k+1}.$

III. {\it Vanishing condition.}\; Notice that $1\in V_{\varphi_{i}^{k}}(z_{i}^{k})\cap V_{\varphi_{j}^{k+1}}(z_{i}^{k+1}).$
Then,
\be
\int_{\mathbb{B}_n} \big [ f-\mathcal{P}_{\varphi_{i}^{k}}(f) \big ]  \varphi_{i}^{k} d v_{\alpha} =0
\ee
and
\be
\int_{\mathbb{B}_n} \bigg ( \big [ f-\mathcal{P}_{\varphi_{j}^{k+1}}(f) \big ] \varphi_{i}^{k}-
{P}_{\varphi_{j}^{k+1}} \Big ( \big [ f-\mathcal{P}_{\varphi_{j}^{k+1}}(f) \big ] \varphi_{i}^{k} \Big ) \bigg) \varphi_{j}^{k+1}dv_{\alpha}=0.
\ee
Therefore, $\int_{\mathbb{B}_n}b_{i}^{k}dv_{\alpha}=0$ and so,
\be
\int_{\mathbb{B}_n} a_{i}^{k}dv_{\alpha}=0.
\ee

IV. {\it Moment condition.}\; We need to estimate
\be
\left|\int_{\mathbb{B}_n}b_{i}^{k}(z)\Phi(z)dv_{\alpha}(z)\right|
\ee
in terms of $\| \Phi \|_{\mathcal{S}_{N}(B^{\varrho}(z_{i}^{k},C r_{i}^{k}))}$ for any $\Phi\in\mathcal{C}^\infty(B^{\varrho}(z_{i}^{k},Cr_{i}^{k})).$ To this end, we first note that there exists a unitary operator $U_{z^k_i}$ such that $U_{z^k_i} z^k_i = (|z^k_i|, 0, \ldots, 0).$ For any $z \in B^{\varrho}(z_{i}^{k},Cr_{i}^{k})$ we assume
\be
U_{z^k_i} z = (x_1 + \mathrm{i} y_1, \ldots, x_n + \mathrm{i} y_n).
\ee
Then, by Lemma \ref{le:Psedu-metric03} we have
\be
\big | x_1 + \mathrm{i} y_1 - |z^k_i| \big | \le \varrho (z^k_i, z)\quad \text{and}\quad \sum^n_{j=2} |x_j + \mathrm{i} y_j|^2 \le 2 \varrho (z^k_i, z).
\ee
This yields that $|x_1 - |z^k_i||, |y_1| \lesssim r^k_i$ and $|x_j|, |y_j| \lesssim \sqrt{r^k_i}$ for $j \ge 2.$ Thus, \be
|\alpha_1|, |\beta_1| \lesssim r^k_i\quad \text{and}\quad |\alpha_j|, |\beta_j| \lesssim \sqrt{r^k_i}\; \text{for}\; j \ge 2
\ee
if $\Theta (z^k_i, z) = (\alpha_1, \beta_1, \ldots, \alpha_n, \beta_n).$

Using local coordinates $(\alpha_1, \beta_1, \ldots, \alpha_n, \beta_n),$
we denote by $P^{\Phi}_{z^{i}_{k},N}$ the Taylor expansion of order $N -1$ of $\Phi$ around $z^k_i$ on $B^{\varrho}(z_{i}^{k},Cr_{i}^{k}),$ i.e.,
\be
P^{\Phi}_{z^{i}_{k},N }(z)=\sum_{|J|\leq N -1}c_{J}  \frac{\partial^{k_1 + \cdots + k_n + m_1 + \cdots + m_n} \Phi }{\partial \alpha^{k_1}_1 \partial \beta^{m_1}_1 \cdots \partial \alpha^{k_n}_n \partial \beta^{m_n}_n} (z_{i}^{k}) \Theta (z_{i}^{k},z)^{J},
\ee
where $J = (k_1, m_1, \ldots, k_n, m_n).$ Note that $P^{\Phi}_{z^{i}_{k},N}$ is in $V_{\varphi_{i}^{k}}(z_{i}^{k}).$ Then, we have
\be
\left\|\Phi-P^{\Phi}_{z^{i}_{k},N}\right\|_{L^{\infty}(B^{\varrho}(z_{i}^{k},Cr_{i}^{k}))} \lesssim \left\|\Phi\right\|_{\mathcal{S}_N (B^{\varrho}(z_{i}^{k},C r_{i}^{k}))}.
\ee
In addition, if $B^\varrho (z_{i}^{k},r_{i}^{k})\cap B^\varrho (z_{j}^{k+1},r_{j}^{k+1})\neq\emptyset,$ there exists a constant $C>0$ such that $r^{k+1} \le C r^k_i$ and so $B^\varrho (z_{j}^{k+1},r_{j}^{k+1}) \subset B^\varrho (z_{i}^{k}, C r_{i}^{k}).$ In this case,  for all $|J|\leq N -1$ we make the change of variable such that $\Theta (z_{i}^{k},z)^{J}$ becomes the element in $V_{\varphi_{j}^{k+1}}(z_{j}^{k+1})$ and its order is still less than $N -1.$ Therefore,
\be\begin{split}
\left|\int_{\mathbb{B}_n}b_{i}^{k}(z)\Phi(z)dv_{\alpha}(z)\right|=&\left|\int_{\mathbb{B}_n}b_{i}^{k}(z)
\left(\Phi(z)-P^{\Phi}_{z^{i}_{k},N}\right)dv_{\alpha}(z)\right|\\
\leq&2^{k_{0}+k+1}v_{\alpha}(B^{\varrho}(z_{i}^{k},C r_{i}^{k}))\left\|\Phi-P^{\Phi}_{z_i^k, N}\right\|_{L^{\infty}(B^{\varrho}(z_{i}^{k},Cr_{i}^{k}))}\\
\lesssim & 2^{k_{0}+k+1}v_{\alpha}(B^{\varrho}(z_{i}^{k},C r_{i}^{k})) \left\|\Phi\right\|_{\mathcal{S}_N (B^{\varrho}(z_{i}^{k},C r_{i}^{k}))},
\end{split}\ee
by the size estimate for $b_{i}^{k}$'s as above. Thus,
\be
\left|\int_{\mathbb{B}^n}a_i^k (z)\Phi(z) d v_{\alpha}(z)\right| \lesssim \|\Phi\|_{\mathcal{S}_N (B^{\varrho}(z_i^k, C r^k_i ))} v_{\alpha}(B^{\varrho}(z_i^k , C r_i^k))^{1-\frac{1}{p}}.
\ee
and so $a^k_i$ is a $(p, \8, N)_{\alpha}$-atom.

V. {\it Convergence in the sense of distributions.}\; In the following, we will show that \eqref{eq:atomicdecomp} holds in the sense of distributions. It suffices to verify that $\sum_{k=0}^{m} \sum^{\8}_{i=1} b^k_i$ convergence in the sense of distributions. Let $\Psi\in C^{\infty}(\mathbb{B}_n).$ We have for any $m >n,$
\be\begin{split}
\bigg | \int_{\mathbb{B}_n} & \sum_{k=n}^{m} \sum^{\8}_{i=1} b^k_i \Psi dv_{\alpha} \bigg |
\le \sum_{k=n}^{m}\sum_{i=1}^{\infty} \bigg |\int_{\mathbb{B}_n}b_{i}^{k}\Psi dv_{\alpha} \bigg |\\
\leq & \sum_{k=n}^{m}\sum_{i=1}^{\infty}2^{k_{0}+k+1}v_{\alpha}(B^{\varrho}(z_{i}^{k},C r_{i}^{k}))\left\|\Psi\right\|_{\mathcal{S}_{N}(B^{\varrho}(z_{i}^{k},C r_{i}^{k}))}\\
\lesssim & \sum_{k=n}^{m}\sum_{i=1}^{\infty}
2^{k_{0}+k+1}v_{\alpha} \big(B^{\varrho}(z_{i}^{k}, C r_{i}^{k})\big) \big(r_{i}^{k}\big)^{\frac{N}{2}}\left\|\Psi\right\|_{C^{N} (\mathbb{B}_n)}\\
\lesssim&\sum_{k=n}^{m}\sum_{i=1}^{\infty}2^{k_{0}+k+1}
v_{\alpha}\big(B^{\varrho}(z_{i}^{k}, C r_{i}^{k})\big)^{\frac{1}{p}} \big(r_{i}^{k}\big)^{\frac{N}{2}-  (\frac{1}{p}-1) \max \{ n+1, n+1+ \alpha \}} \left \|\Psi\right\|_{C^N (\mathbb{B}_n)}\\
\lesssim & \sum_{k=n}^{m}\sum_{i=1}^{\infty}2^{k_{0}+k+1}
v_{\alpha}\big(B^\varrho (z_{i}^{k},r_{i}^{k})\big)^{\frac{1}{p}}\left\|\Psi\right\|_{C^N (\mathbb{B}_n)},
\end{split}\ee
where we use the fact $N \ge  N_{p, \alpha}$ and Lemma \ref{le:Psedu-metric02}. Hence, since $\frac{1}{p}\geq1$ we have
\be\begin{split}
\bigg | \int_{\mathbb{B}_n} &  \sum_{k=n}^{m} \sum^{\8}_{i=1} b^k_i (z) \Psi(z) dv_{\alpha}(z) \bigg |\\
\lesssim & \sum_{k=n}^{m}2^{k_{0}+k+1}\bigg(\sum_{i=1}^{\infty}
v_{\alpha}\big(B^\varrho (z_{i}^{k},r_{i}^{k})\big)\bigg)^{\frac{1}{p}}\left\|\Psi\right\|_{C^N (\mathbb{B}_n)}\\
\lesssim & N_{0}\sum_{k=n}^{m}2^{k_{0}+k+1}v_{\alpha}(\mathcal{O}_k)^{\frac{1}{p}}\left\|\Psi\right\|_{C^N (\mathbb{B}_n)}\\
\leq&N_{0}\left(\sum_{k=n}^{m}2^{(k_{0}+k+1)p}v_{\alpha}(\mathcal{O}_k)\right)^{\frac{1}{p}}
\left\|\Psi\right\|_{C^N(\mathbb{B}_n)}\\
\lesssim & \left\|\Psi\right\|_{C^N (\mathbb{B}_n)} \left(\sum_{k=n}^{m}\int_{2^{k_{0}+k-1}}^{2^{k_{0}+k}}t^{p-1}v_{\alpha}
\left\{z \in \mathbb{B}_n:\; \mathcal{K}_{\mu,L}(f)(z)+f^{\star}_{\delta}(z)>t\right\}dt\right)^{\frac{1}{p}}\\
=&\left(\int_{\mathcal{O}_{n-1}}|\mathcal{K}_{\mu,L}(f)+f^{\star}_{\delta}|^{p}dv_{\alpha}\right)^{\frac{1}{p}}
\left\|\Psi\right\|_{C^N (\mathbb{B}_n)},
\end{split}\ee
which tends to $0$ as $n \rightarrow \infty.$ Thus, the equality \eqref{eq:atomicdecomp} holds in the sense of distributions.

VI. {\it Coefficients in $\el^{p}.$}\; At last, we only need to check the coefficients $\{\lambda_{i}^{k}\}$ belong to $\el^{p}.$ Indeed,
\be\begin{split}
\sum_{k=0}^{\infty}\sum_{i=1}^{\infty}|\lambda_{i}^{k}|^{p}&=\sum_{k=0}^{\infty}\sum_{i=1}^{\infty}2^{(k_{0}+k+1)p}v_{\alpha}(B^{\varrho}(z_{i}^{k},C r_{i}^{k}))\\
& \lesssim \sum_{k=0}^{\infty}\sum_{i=1}^{\infty}2^{(k_{0}+k+1)p}v_{\alpha}(B^\varrho (z_{i}^{k},r_{i}^{k})) \lesssim \sum_{k=0}^{\infty}2^{(k_{0}+k+1)p}v_{\alpha}(\mathcal{O}_{k})\\
& \lesssim\int_{2^{k_{0}}}^{\infty}t^{p-1}v_{\alpha}
\left\{z \in \mathbb{B}_n:\; \mathcal{K}_{\mu,L}(f)(z)+f^{\star}_{\delta}(z)>t\right\}dt\\
& \lesssim\|\mathcal{K}_{\mu,L}(f)+f^{\star}_{\delta}\|_{p,\alpha}^{p} \lesssim \|f\|_{p,\alpha}^{p}.
\end{split}\ee

\subsection{Proof of Theorem \ref{th:atomic-deco}}\label{AtomicdecompPf}

By far, we have shown that
\be
f = \lambda_{0}a_{0}+\sum_{k=0}^{\infty}\sum_{i=1}^{\infty}\lambda_{i}^{k}a_{i}^{k}
\ee
holds true in the distribution sense and hence the assertion (1) of Theorem \ref{th:atomic-deco} holds. Consequently, by the assertion (1) in Theorem \ref{th:p-atomBergman} we conclude that the series on the right hand side of \eqref{eq:atomicdecomp} converges in $L^p_{\alpha} (\mathbb{B}_n)$ and so the assertion (2) of Theorem \ref{th:atomic-deco} holds true as well.

It remains to prove the assertions (3) and (4) in Theorem \ref{th:atomic-deco}. Indeed, assuming that $f\in\mathcal{A}_{\alpha}^{p}\cap\mathcal{A}_{\alpha}^{2},$ we have
\be\begin{split}
f(z)=&P_{\alpha}(f)(z)\\
=& \bigg \langle \lambda_{0}a_{0}+ \sum_{k=0}^{\infty}\sum_{i=1}^{\infty}\lambda_{i}^{k}a_{i}^{k}, K^{\alpha}(\cdot,z) \bigg \rangle\\
=& \lambda_{0}P_{\alpha}(a_{0})(z)+ \bigg \langle \sum_{k=0}^{\infty}\sum_{i=1}^{\infty}\lambda_{i}^{k}a_{i}^{k}, K^{\alpha}(\cdot,z) \bigg \rangle\\
=& \lambda_{0}P_{\alpha}(a_{0})(z)+\sum_{k=0}^{\infty}\sum_{i=1}^{\infty}\lambda_{i}^{k}P_{\alpha}(a_{i}^{k})(z),
\end{split}\ee
since the series converges in the sense of distributions. Moreover,
\be
|\lambda_{0}|^{p}+\sum_{k=0}^{\infty}\sum_{i=1}^{\infty}|\lambda_{i}^{k}|^{p}\lesssim\|f\|_{p,\alpha}^p.
\ee
Therefore, the assertions (3) and (4) of Theorem \ref{th:atomic-deco} hold true for $f\in\mathcal{A}_{\alpha}^{p}\cap\mathcal{A}_{\alpha}^{2}.$

We next show that they hold for any $f$ in $\mathcal{A}_{\alpha}^{p}$ ($0<p\leq1$). Since $\mathcal{A}_{\alpha}^{p}\cap\mathcal{A}_{\alpha}^{2}$ is dense in $\mathcal{A}_{\alpha}^{p},$ we can choose a sequence $\{f_n\}_{n=1}^{\infty}$ in $\mathcal{A}_{\alpha}^{p}\cap\mathcal{A}_{\alpha}^{2}$ such that
\be
f=f_1+\sum_{n=2}^{\infty}(f_{n}-f_{n-1}), \quad
\|f_n-f\|_{p,\alpha}^{p} \rightarrow 0 \qquad \mathrm{as}\ n\rightarrow\infty,
\ee
and
\be
\|f_1\|_{p,\alpha}^{p}\leq\frac{3}{2}\|f\|_{p,\alpha}^{p}, \qquad \|f_n-f_{n-1}\|_{p,\alpha}^{p} \leq \frac{1}{2^{n}}\|f\|_{p,\alpha}^{p},\; \forall n>1.
\ee
Let $g_1=f_1$ and $g_n=f_n-f_{n-1}$ for $n>1.$ Since $g_n\in\mathcal{A}_{\alpha}^{p}\cap\mathcal{A}_{\alpha}^{2}(n\geq1),$ we can choose numbers $\lambda_{i}^{n}$ and $(p, \8, N)_{\alpha}$-atoms $a_{i}^{n}$ such that
\be
g_n=\sum_{i=1}^{\infty}\lambda_{i}^{n}P_{\alpha}(a_{i}^{n}),\ \ \sum_{i=1}^{\infty}|\lambda_{i}^{n}|^{p}\lesssim\|g_n\|_{p,\alpha}^{p}.
\ee
Hence,
\be
f=\sum_{n=1}^{\infty}g_{n}=\sum_{n=1}^{\infty}\sum_{i=1}^{\infty}\lambda_{i}^{n}P_{\alpha}(a_{i}^n)
\ee
and
\be
\sum_{n=1}^{\infty}\sum_{i=1}^{\infty}|\lambda_{i}^{n}|^{p}\lesssim\sum_{n=1}^{\infty}\|g_n\|_{p,\alpha}^{p}\lesssim\|f\|_{p,\alpha}^{p}.
\ee
This completes the proof of Theorem \ref{th:atomic-deco}.
\hfill$\Box$

\subsection{The case of $(p, q)_{\alpha}$-atoms}\label{(p,q)-atom}

The previous results can be extended to the case of $(p, q)_{\alpha}$-atoms for $1 < q < \8.$

\begin{definition}\label{df:(p,q)-atom}
Let $0<p \le 1 \le q \le \8$ and $p < q.$ Let $\alpha > -1.$ Let $N \ge N_{p, \alpha}$ be an integer. A measurable function $a$ on $\mathbb{B}_n$ is a $(p, q, N)_{\alpha}$-atom if it is either a bounded function with $\|a\|_{L^\infty}\leq1,$ or satisfies that there exist $z_0 \in \mathbb{B}^n$ and $r_0 >0$ such that
\begin{enumerate}[{\rm (1)}]
\item $a$ is supported in $B^{\varrho}(z_{0},r_{0});$

\item $\|a \|_{q, \alpha} \le v_{\alpha} (B^{\varrho}(z_{0},r_{0}))^{\frac{1}{q} -\frac{1}{p}};$

\item $\int_{\mathbb{B}_n}a(z)\,dv_{\alpha}(z) = 0;$

\item for all $\phi\in\mathcal{C}^\infty(B^{\varrho}(z_{0},r_{0}))$
\be
\left|\int_{\mathbb{B}^n}a(z)\phi(z)dv_{\alpha}(z)\right|\le
\|\phi\|_{\mathcal{S}_N (B^{\varrho}(z_{0},r_{0}))}
v_{\alpha}(B^{\varrho}(z_{0},r_{0}))^{1-\frac{1}{p}}.
\ee

\end{enumerate}
\end{definition}

By slightly modifying the proof of Theorem \ref{th:p-atomBergman}, we can prove that Theorem \ref{th:p-atomBergman} is still valid for $(p, q, N)_{\alpha}$-atoms.

\begin{theorem}\label{th:(p,q)-atomBergman}
Let $0<p \le 1 \le q < \8$ and $p<q.$ Let $\alpha > -1.$ Let $N \ge N_{p, \alpha}$ be an integer.
\begin{enumerate}[{\rm (1)}]

\item For any $(p, q, N)_{\alpha}$-atom $a,$ $\|a \|_{p, \alpha} \le 1.$

\item If $q >1,$ there is a constant $C>0$ depending only on $p, n,\alpha, q$ and $N$ such that
\be
\left \|P_{\alpha}(a) \right \|_{p, \alpha} \leq C
\ee
for any $(p, q, N)_{\alpha}$-atom $a.$

\item Let $q>1.$ If $\{a_k\}$ is a sequence of $(p, q, N)_{\alpha}$-atoms, then for any sequence $\{\lambda_k\}$ of complex numbers with $\sum_k |\lambda_k|^p < \8,$ the series $\sum_k \lambda_k P_{\alpha} (a_k)$ converges unconditionally in $\mathcal{A}^p_{\alpha}$ such that
\be
f : = \sum_k \lambda_k P_{\alpha} (a_k) \in \mathcal{A}^p_{\alpha} \quad \text{and} \quad \| f \|^p_{p , \alpha} \lesssim \sum_k |\lambda_k|^p.
\ee

\end{enumerate}
\end{theorem}

On the other hand, it is easy to check that $(p, \8, N)_{\alpha}$-atoms are necessarily $(p, q, N)_{\alpha}$-atoms. Therefore, by Theorem \ref{th:atomic-deco} we have

\begin{theorem}\label{th:(p,q)-atomicdeco}
Let $0<p \le 1 \le q < \8$ and $p<q.$ Let $\alpha > -1.$ Let $N \ge N_{p, \alpha}$ be an integer. For each $f \in \mathcal{A}^p_{\alpha}$ there exist a sequence of complex numbers $\{ \lambda_j \}$ with $\sum_j |\lambda_j |^p < \8,$ and a sequence of $(p,q, N)_{\alpha}$-atoms $\{a_j\}$ such that
\begin{enumerate}[{\rm (1)}]

\item $f = \sum_j \lambda_j a_j$ in the sense of distributions;

\item $f = \sum_j \lambda_j a_j,$ where the series $\sum_j \lambda_j a_j$ converges unconditionally in $L^p_{\alpha} (\mathbb{B}_n);$

\item if $q>1,$ $f = \sum_j \lambda_j P_{\alpha}(a_j),$ where the series $\sum_j \lambda_j P_{\alpha} (a_j)$ converges unconditionally in $\A^p_{\alpha};$

\item $\sum_j |\lambda_j |^p \lesssim \| f \|^p_{p , \alpha}.$

\end{enumerate}

Consequently, for any $f \in \mathcal{A}^p_{\alpha}$ one has
\be
\| f \|^p_{p , \alpha} \approx \inf \bigg \{ \sum_j |\lambda_j |^p:\; f = \sum_j \lambda_j P_{\alpha} (a_j) \bigg \}
\ee
if $q>1,$ where the infimum is taken over all decompositions of $f$ described above.

\end{theorem}

\subsection*{Acknowledgement} This paper is part of the second named author's Ph.D thesis, under the supervision of the first named author. This research was supported in part by the NSFC under Grant No. 11171338.

\end{document}